\documentclass[12pt,twoside,reqno]{amsart}

\usepackage{amssymb,amsmath,amstext,amsthm,amsfonts,amscd,xcolor}
\usepackage{mathtools}

\usepackage{dsfont}

\usepackage{amsthm, enumerate}

\usepackage[ansinew]{inputenc} 
\usepackage{graphicx}
\usepackage[mathscr]{eucal}
\usepackage{hyperref}
\allowdisplaybreaks

\newcommand{\R}{\mathbb{R}}

\newcommand{\C}{\mathbb{C}}
\newcommand{\N}{\mathbb{N}}
\newcommand{\Z}{\mathbb{Z}}

\newcommand{\T}{\mathbb{T}}

\newcommand{\Pp}{\mathbb{P}}

\newcommand{\EE}{\mathbb{E}}
\newcommand{\Escr}{\mathscr{E}}

\theoremstyle{plain}
\newtheorem{theorem}{Theorem}[section]
\newtheorem{proposition}{Proposition}[section]

\newtheorem{lemma}[proposition]{Lemma}
\theoremstyle{definition}

\newtheorem{definition}{Definition}[section]
\newtheorem*{theorem*}{Theorem}
 
\theoremstyle{definition}
\newtheorem{remark}{Remark}[section]

\numberwithin{equation}{section}

\newcommand{\abs}[1]{\left| #1 \right|} 
\newcommand{\norm}[1]{\lVert#1\rVert} 
\newcommand{\bnorm}[1]{\Bigl\| #1\Bigr\|} 
\newcommand{\normtwo}[1]{
{\left\vert\kern-0.25ex\left\vert\kern-0.25ex\left\vert #1 
    \right\vert\kern-0.25ex\right\vert\kern-0.25ex\right\vert} }

\newcommand{\ep}{\epsilon} 
  
\newcommand{\la}{\lambda}
\newcommand{\ga}{\gamma}

\newcommand{\om}{\omega}

\makeatletter
\newsavebox\myboxA
\newsavebox\myboxB
\newlength\mylenA

\newcommand*\xoverline[2][0.75]{%
    \sbox{\myboxA}{$\m@th#2$}%
    \setbox\myboxB\null
    \ht\myboxB=\ht\myboxA%
    \dp\myboxB=\dp\myboxA%
    \wd\myboxB=#1\wd\myboxA
    \sbox\myboxB{$\m@th\overline{\copy\myboxB}$}
    \setlength\mylenA{\the\wd\myboxA}
    \addtolength\mylenA{-\the\wd\myboxB}%
    \ifdim\wd\myboxB<\wd\myboxA%
       \rlap{\hskip 0.5\mylenA\usebox\myboxB}{\usebox\myboxA}%
    \else
        \hskip -0.5\mylenA\rlap{\usebox\myboxA}{\hskip 0.5\mylenA\usebox\myboxB}%
    \fi}
\makeatother

\newcommand{\nldt}{\xoverline{n}}
\newcommand{\cldt}{\xoverline{c}}

\newcommand{\ind}{\textbf{1}}

\newcommand{\B}{\mathscr{B}}

\newcommand{\id}{{\rm id}}

\newcommand\restr[2]{{
  \left.\kern-\nulldelimiterspace 
  #1 
  \vphantom{\big|} 
  \right|_{#2} 
  }}

\newcommand{\Hscr}{\mathscr{H}}

\newcommand{\Prob}{\mathrm{Prob}}

\newcommand{\supp}{\mathrm{supp}}

\newcommand{\Qop}{\mathcal{Q}}

\newcommand{\smu}{\nu}

\title[Statistical properties for mixing Markov chains]{Statistical properties for mixing Markov chains with applications to \\dynamical systems}

\date{}

\begin{document}

\author[A. Cai]{Ao Cai}
\address{School of Mathematical Sciences, Soochow University, Soochow 215006, China and Departamento de Matem\'atica, Pontif\'icia Universidade Cat\'olica do Rio de Janeiro (PUC-Rio), Brazil}
\email{godcaiao@gmail.com}

\author[P. Duarte]{Pedro Duarte}
\address{Departamento de Matem\'atica and CMAFcIO\\
Faculdade de Ci\^encias\\
Universidade de Lisboa\\
Portugal 
}
\email{pmduarte@fc.ul.pt}

\author[S. Klein]{Silvius Klein}
\address{Departamento de Matem\'atica, Pontif\'icia Universidade Cat\'olica do Rio de Janeiro (PUC-Rio), Brazil}
\email{silviusk@puc-rio.br}

\begin{abstract}   
We establish an abstract, effective, exponential large deviations type estimate for Markov systems satisfying a weaker form of mixing. We employ this result to derive such estimates, as well as a central limit theorem, for the skew product encoding a random torus translation, a model we call a mixed random-quasiperiodic dynamical system. This abstract scheme is applicable to many other types of skew product dynamics, including systems for which the spectral gap property for the transition or the transfer operator does not hold.  
\end{abstract}



\maketitle


\section{Introduction and statements}\label{intro}
Let $M$ be a Polish metric space. 
A (deterministic) dynamical system on $M$ is a continuous function $f \colon M \to M$ that encodes a law of  transitioning from a state $x \in M$ to the next state $f (x) \in M$. A Borel probability measure $\smu$ on $M$ is called $f$-invariant if for any Borel set $E \subset M$,
$$\smu (E) = \smu \left( f^{- 1} (E) \right) = \int_M \delta_{f (x)} (E) \, d \smu (x) \, .$$ 
In this case, the triplet $(M, f, \smu)$ is called a measure preserving dynamical system (MPDS). We also assume that this system is ergodic, that is, if a Borel set $E \subset M$ is $f$-invariant (meaning that $f^{-1} (E) = E$) then $\smu (E)$ is equal to $0$ or $1$.

Let $\Escr$ be a set of observables $\varphi \colon M \to \R$ which we generally assume to be a Banach subspace of $L^\infty (M)$, the space of  bounded, measurable functions on $M$. Given $\varphi \in \Escr$ and $n\in \N$, let
$$S_n \varphi := \varphi + \varphi \circ f + \cdots + \varphi \circ f^{n-1}$$
denote the corresponding $n$-th Birkhoff sum. By the pointwise ergodic theorem, the Birkhoff time averages converge a.e. to the space average: 
\begin{equation}\label{Birkhoff}
\frac{1}{n} \, S_n \varphi \to \int_M \varphi \, d \smu \quad \text{as} \quad n\to\infty, \quad \smu-\text{a.e.}
\end{equation}

This is the analogue of the law of large numbers in probabilities. We are interested in other types of statistical properties such as large deviations type (LDT) estimates or a central limit theorem (CLT) for certain types of dynamical systems. 

The $\smu$-almost everywhere convergence~\eqref{Birkhoff} implies the convergence in measure, that is, for all $\ep>0$,
$$ \smu \Big\{ x \in M \colon \Big| \frac{1}{n} \, S_n \, \varphi (x) - \int_M \, \varphi \, d \smu  \Big| > \epsilon \Big\}   \to 0 \quad \text{ as } \ n \to \infty  .$$
When the rate of convergence to zero of the exceptional set of states is explicit, for instance exponential, we say that the MPDS $(M, f, \smu)$ satisfies an LDT estimate on $\Escr$. There are different kinds of LDT estimates: asymptotic, i.e. in the spirit of the classical large deviations principle of Cram\'er, or finitary, in the spirit of Hoeffding's inequality. We are more interested in the latter, as it is more effective and it has applications to other topics in dynamics, see~\cite{DK-book}. 

Let $X_1, \ldots, X_n$ be independent random variables and denote by $S_n := X_1 + \cdots + X_n$ their sum. Hoeffding's inequality states that if $\abs{X_i} \le C$ a.s. for all $1\le i \le n$,  then for all $\ep>0$ we have
\begin{equation*}
  \Pp \Big\{\, \Big| \frac{1}{n}\, S_n - \EE \big(\frac{1}{n}\, S_n\big)  \Big| > \ep \,\Big\}   \le 2 \,  e^{- c (\ep) \, n} \, ,
\end{equation*}  
 where $c (\ep) = (2 C)^{-2}   \, \ep^2$.
   
Hoeffding-type inequalities are available for a wide class of non-uniformly hyperbolic dynamical systems and for a large space of observables, see for instance I. Melbourne, M. Nicol~\cite{Melbourne-Nicol}, I. Melbourne~\cite{Melbourne-PAMS} and J. F. Alves, J. M.  Freitas, S. Luzzatto, S. Vaienti~\cite{Alves-Freitas-L-V}. Moreover, these works provide general criteria for the existence of LDT estimates with various decaying rates depending on the availability of an appropriate decay of correlations, see~\cite[Theorem D and Theorem E]{Alves-Freitas-L-V}.

Furthermore, J.-R. Chazottes and S. Gou\"ezel~\cite{Chazottes-concentration} strengthened these types of result by establishing concentration inequalities for dynamical systems modeled by (various types of) Young towers.
Decay of correlations and a central limit theorem in these settings were previously obtained by L.S. Young~\cite{LSYoung-annals, LSYoung99}. There is a vast body of work on these topics, which we will not attempt to review, but we recommend to the interested reader the monographs~\cite{Baladi} and~\cite{Liverani}.

In another important and relevant monograph~\cite{HH} by H. Hennion and L. Herv\'e, the authors developed a general functional analytic method for establishing limit theorems (including an asymptotic large deviations principle and a central limit theorem) for  strongly mixing (in an appropriate sense) Markov chains. This strong mixing property implies the quasi-compactness (on an appropriate space of observables) of the Markov operator determined by the transition kernel of the Markov chain. The quasi-compactness is then inherited, using the perturbation theory of linear operators, by nearby elements of a one-parameter family of Laplace-Markov operators, which  is then used to prove the limit theorems for the chain.  This abstract scheme is employed to establish limit laws for products of invertible random matrices satisfying some generic conditions (as well as for other systems such as uniformly expanding maps).  
We would like to note that this monograph greatly influenced some of our previous works and motivated the present one.

\medskip

The first result of this paper is an abstract, effective, exponential LDT estimate for Markov chains satisfying a rather {\em weak} form of strong mixing. In particular, the quasi-compactness of the corresponding Markov operator may fail to hold (which will indeed be the case in certain interesting applications). Besides being more general, compared to~\cite{HH}, our method is much more straightforward and does not use the perturbation theory of linear operators.
We will subsequently apply this general result to some skew-product dynamical systems. 

\medskip

Let us briefly describe this abstract setting (see Section~\ref{aldts} for more details).
A stochastic dynamical system (SDS) on $M$ (also referred to as a Markov kernel) is a continuous function $K \colon M \to \Prob (M)$, where the set $\Prob (M)$ of probability measures is 
equipped with the weak topology (see~\cite[Section 2.1]{CDK-paper1} for precise definitions). 
For a point $x \in M$ and a Borel set $E\subset M$, $K_x (E)$ can be interpreted as the probability that the state $x$ transitions to some state in $E$. A measure $\smu \in \Prob (M)$ is called $K$-stationary if for any Borel set $E \subset M$,
$$\smu (E)  = \int_M K_x (E) \, d \smu (x) \, .$$ 
In this case, the triplet $(M, K, \smu)$ is called a Markov system.

\medskip

The Markov (or transition) operator $\Qop = \Qop_K$ associated to a Markov system $(M, K, \smu)$ is a priori defined on $L^\infty (M)$ by
$$\Qop \varphi (x) := \int_M \varphi (y) \, d K_x (y) \quad \forall x \in M .$$

Let  $\Escr$ be a $\Qop$-invariant Banach space of observables containing the constant function $\ind$ and such that $\Escr \hookrightarrow L^\infty (M)$ is a continuous embedding. 

\medskip

Finally, let $\{Z_n\}_{n\geq 0}$ be a $K$-Markov chain, that is, a Markov chain with values in $M$ and transition kernel $K$.  Given an observable $\varphi \colon M \to \R$ and an integer $n\in\N$, denote by  
$$
S_n \varphi :=   \varphi \circ Z_0 + \cdots + \varphi \circ Z_{n-1}
$$
the corresponding ``stochastic'' Birkhoff sums.\footnote{The notation $S_n \varphi$ will be used for both deterministic and stochastic Birkhoff sums. It will be clear from the context if we mean one or the other.}

\medskip

Evidently, any deterministic dynamical system $(M, f)$ is also stochastic, where the Markov kernel $K$ is just $K_x = \delta_{f(x)}$ for all $x \in M$, any $f$-invariant measure $\smu$ is $K$-stationary and the Koopman operator is the corresponding Markov operator. However, the Koopman operator is in general not strongly mixing, a key property usually needed to establish statistical properties of the system. We will then associate to a non invertible dynamical system $(M, f)$ a different SDS (consider for instance the transition kernel that assigns to any $x \in M$ its weighted pre-images via $f$) which will be strongly mixing in an appropriate sense.

\medskip

Mixing in general refers to the convergence
\begin{equation}\label{mixing gen}
K_x^n \to \smu \quad \text{as } n \to \infty
\end{equation}
of the $n$-th convolution power of the Markov kernel $K$ to the stationary measure $\smu$. 

This is equivalent to the convergence of the powers of the Markov operator to the linear functional determined by the stationary measure:
\begin{equation}\label{mixing gen2}
\Qop^n \varphi \to \int_M \varphi \, d\smu \quad \text{as } n \to \infty
\end{equation}
for all $\varphi$ in an appropriate space of observables. 

The convergence above can be understood in different ways, whence the different types of mixing. 

The strongest form of mixing, in general referred to as uniform ergodicity (see~\cite[Chapter 16]{Markov-chains-book}) assumes in~\eqref{mixing gen} the uniform convergence in $x$ with respect to the total variation norm. The rate of uniform convergence is then necessarily exponential. Moreover, it is equivalent to the exponential rate of convergence 
in~\eqref{mixing gen2} for any observable in $L^\infty (M)$. The Hoeffding inequality is already available in this context, see~\cite{Hoeffding-Markov}.  

A weaker type of mixing, which we refer to as spectral strong mixing on $(\Escr,  \norm{ \cdot }_\Escr)$ is defined as follows: for all $\varphi \in \Escr$ and $n \in \N$
$$ \norm{\Qop^n \varphi-\int_M \varphi\, d\smu}_\Escr \leq C \,\sigma^n\, \norm{\varphi}_\Escr  $$
for some constants $C<\infty$ and $\sigma \in (0, 1)$. This is equivalent to  a  spectral gap  for the Markov operator  $\Qop \colon \Escr\to\Escr$, between the simple eigenvalue $1$ and the rest of the spectrum of $\Qop$, which is contained in the unit disk (thus $\Qop$ is quasi-compact and simple on $\Escr$). The spectral gap property is widely used to establish statistical properties for dynamical systems (see~\cite[Chapter 5]{DK-book} for its use in conjunction with the Markov operator and~\cite{Baladi, Liverani} for the transfer operator).

\medskip

We introduce the following weaker version of strong mixing.

\begin{definition}\label{strong mixing intro} Let $\{r_n\}_{n\in\N}$ be a decreasing sequence with $r_n \to 0$ as $n\to\infty$.
A Markov system $(M, K,\smu)$  is  called strongly mixing on $(\Escr,  \norm{ \cdot }_\Escr)$ with mixing rate $\{r_n\}_{n\in\N}$  if
for all $\varphi\in\Escr$ and  $n\in\N$,
\begin{equation}\label{mixing condition intro}
 \norm{\Qop^n \varphi-\int_M \varphi\, d\smu}_\infty \leq r_n \, \norm{\varphi}_\Escr  \, . 
 \end{equation}
In this case we also refer to  the restriction of the Markov operator $\Qop$ to the space of observables $\Escr$ as being strongly mixing.
\end{definition}

Notice the distinct r\^oles played by the  norms $\norm{\cdot}_\infty$ and  $\norm{\cdot}_\Escr$ in this definition compared to the spectral form of strong mixing introduced earlier: as the inclusion $\Escr \xhookrightarrow{}  L^\infty (M)$ is assumed bounded, the $\norm{\cdot}_\Escr$-norm is stronger than the $\norm{\cdot}_\infty$-norm. Another way in which this concept is more general is the arbitrary (hence potentially much weaker than exponential) rate of convergence to the stationary measure. For related concepts and sets of results see J. Dedecker, S. Gou\"ezel, F. Merlev\`ede~\cite{slow-mixing}, M. Peligrad, S. Utev, W. B. Wu~\cite{Peligrad}, J. Dedecker, C. Prieur~\cite{Dedecker-Prieur}.

\medskip

We are ready to state the first result of this paper, an abstract, effective, exponential LDT estimate for strongly mixing Markov chains.

\begin{theorem}\label{abstract ldt intro}
Let $(M, K,\smu)$  be a strongly mixing Markov system on $(\Escr,  \norm{ \cdot }_\Escr)$ with mixing rate $r_n \searrow 0$ and let $\{Z_n\}_{n\geq 0}$ be a $K$-Markov chain.
Then for all $\varphi\in \Escr$ and $\epsilon >0$  there are $c (\ep) > 0$ and $n (\ep) \in \N$ such that for all $n \ge n (\ep)$ and $x_0 \in M$   we have
$$
\Pp_{x_0} \left\{\abs{\frac{1}{n}S_n\varphi-\int_M \varphi d\smu}>\epsilon\right\}\leq 8 e^{-c(\epsilon)n}
$$
where the parameters $c(\epsilon) $ and $n (\ep)$ depend explicitly and uniformly on the input data, namely on $\epsilon$, $\norm{\varphi}_\Escr$ and the mixing rate. More precisely,  the threshold $n (\ep)$ is the first integer $n_0$ such that $r_{n_0} \le \frac{\ep}{4 \norm{\varphi}_\Escr}$, while the exponential rate is 
$c (\ep) = \frac{1}{8 \, \norm{\varphi}^2_\Escr} \, \frac{\ep^2}{n (\ep)}$.
\end{theorem}

\begin{remark}
If  the mixing rate is a power, that is, $r_n = C \, \frac{1}{n^p}$ for some $p>0$ and $C<\infty$  then a simple calculation shows that $n (\ep) = \nldt \, \ep^{-\frac{1}{p}}$ and $c (\ep) = \cldt \, \ep^{2+\frac{1}{p}}$, where $\nldt = (4 C  \norm{\varphi}_\Escr)^{\frac{1}{p}}$ and $\cldt = \frac{1}{8} (4C)^{- \frac{1}{p}} \, \norm{\varphi}^{ - 2 - \frac{1}{p}}_\Escr$.

If the mixing rate is exponential, that is, $r_n = e^{- c n}$ for some $c>0$, then $n (\ep) \sim \log \frac{1}{\ep}$ while $c (\ep) \sim \ep^2 \left( \log \frac{1}{\ep} \right)^{-1}$. 
\end{remark}

\begin{remark}
We note that when the mixing rate of the Markov operator is of the form $r_n \sim \frac{1}{n^p}$ with $p > \frac 1 2$, the  deviation estimate we obtain in Theorem~\ref{abstract ldt intro} is suboptimal when compared to the one obtained by M. Peligrad, S. Utev, W. B. Wu in ~\cite[Proposition 2]{Peligrad}. More precisely, as indicated above, our rate of exponential decay of the deviation set is  $c (\ep) \sim \ep^{2+\frac{1}{p}}$ as opposed to the optimal exponential rate  $\ep^2$ in~\cite{Peligrad}. 

Moreover, when the mixing rate is slower, for instance still a power  $r_n \sim \frac{1}{n^p}$ but with $p \in (0,  \frac 1 2]$, 
for a deviation size $\ep$ the bound on the measure of the deviation set in our estimate is of order $e^{- \cldt \, \ep^{(2+\frac 1 p)} \, n}$ while in~\cite{Peligrad} it turns out to be of order $e^{- \cldt \, \ep^{2} \, n^{2 p}}$. Thus  a simple calculation shows that our deviation estimate is stronger than the one in~\cite{Peligrad} when the deviation size $\ep$ and the scale $n$ satisfy the relation
$$\ep > \frac{1}{n^{(1-2 p) p}}$$
and it is weaker otherwise. We then conclude the following. 

\smallskip

When it comes to {\em large} deviations estimates in the literal sense (meaning that the deviation size $\ep$ is a {\em fixed}, possibly small number) our result is essentially the same as the one  in~\cite{Peligrad} when the mixing rate is a power $\frac{1}{n^p}$ with $p > \frac 1 2$ and it is strictly stronger for any other mixing rate (since the bound on the measure of the deviation set we get is always exponentially small, $e^{- c n}$, rather than sub-exponentially small, $e^{- c n^b}, b < 1$ as in~\cite{Peligrad}). 

When it comes to {\em moderate} deviations estimates (meaning that the deviation size $\ep = \ep_n \to 0$ at a certain rate as $n \to \infty$), our result is strictly weaker than the one  in~\cite{Peligrad} when the mixing rate $r_n$ is a power $\frac{1}{n^p}$ with $p \ge \frac 1 2$ or when $r_n \sim \frac{1}{n^p}$ with $p < \frac 1 2$ but $\ep_n \to 0$ relatively fast, and it is strictly stronger when the mixing rate is weaker, for instance a power $\frac{1}{n^p}$ with $p <  \frac 1 2$ and at the same time $\ep_n \to 0$ moderately fast.
\end{remark}

\smallskip

We refer to the article by Dedecker, Merlev\`ede, Peligrad and Utev~\cite{DMPU2009} for more precise results on moderate
deviations when the mixing rate is a power $\frac{1}{n^p}$ with $p > \frac 1 2$. Under these conditions, the authors
establish the functional form of the moderate deviation principle (see their Corollary 2).

\begin{remark}
We would like to emphasize the fact that from the perspective of their further applications to other kinds of problems in dynamical systems and mathematical physics---our main motivation for this work--- we are interested primarily in sharp {\em large} deviations under weaker rates of mixing, which is what was accomplished in this paper. 

\smallskip

As an example of such a future application we mention the regularity of the Lyapunov exponents of certain types of linear cocycles. P. Duarte and S. Klein~\cite[Chapter 3]{DK-book} showed that the availability of exponential, effective large deviations for the multiplicative process associated to the iterates of a linear cocycle implies the optimal, H\"older continuity of its limiting quantity, the Lyapunov exponent. A sub-exponential large deviation would still imply the regularity of the Lyapunov exponent, but it would provide a weaker than H\"older modulus of continuity. 
The relevant deviation size needed in this result is a fraction of the Lyapunov exponent, thus a fixed, possibly small, positive number.

It turns out that Theorem~\ref{abstract ldt intro} can be used to derive such large deviations estimates for many  types of linear cocycles; in particular it will be employed to this end in our future  work on linear cocycles constructed over the base dynamical system considered below in this paper.

\smallskip

Another slight advantage of the approach introduced in this paper for proving deviations estimates for mixing Markov chains is that it is quite straightforward and elementary (by not requiring martingale theory, the Azuma-Hoeffding inequality or perturbation theory of linear operators). 
\end{remark}

\begin{remark}
The mixing assumption in Definition~\ref{strong mixing intro} is phrased in this functional analytic language for esthetic reasons and because this is what we actually obtain in concrete applications. However, what is actually needed for an observable $\varphi \in L^\infty (M)$ to satisfy the LDT estimate in Theorem~\ref{abstract ldt intro} is that the bound~\eqref{mixing condition intro} holds just for $\varphi$ (and not necessarily for all elements of the space $\Escr$).

Furthermore, if in Definition~\ref{strong mixing intro} we only require the uniform convergence of $\Qop^n \varphi$ to $\int \varphi d \smu$ (without an explicit rate of convergence), then an LDT estimate for $\varphi$ still holds; however, this will {\em not be an effective} estimate, in the sense that the threshold $n (\ep)$ for its validity and the exponential rate $c (\ep)$ are not explicitly  determined by the input data. It is thus important to distinguish between effective and non effective LDT estimates.  We note that because of its relation to Hoeffding's inequality and its applications to other topics in dynamical systems  or mathematical physics, it is precisely the effectiveness of the LDT estimates that we are seeking in this work.
\end{remark}



\begin{remark}
The notation $\Pp_{x_0}$ refers to the conditional probability of the Markov chain  $\{Z_n\}_{n\geq 0}$ starting from the point $Z_0 = x_0$. Since $x_0 \in M$ is arbitrary, it follows that the LDT estimate above holds for $K$-Markov chains with any initial distribution, including of course the stationary measure $\smu$.
\end{remark}

We apply Theorem~\ref{abstract ldt intro} to the skew product encoding a random torus translation. The ergodicity of this model, which we call a {\em mixed random-quasiperiodic dynamical system}, was studied in~\cite{CDK-paper1}. Let us recall its definition.

Let $\Sigma:=\T^d$ which we regard as a space of symbols and let $\mu\in \Prob(\Sigma)$ be a probability measure.  Consider the space of sequences $X:= \Sigma^\Z$   which we endow with the product measure $\mu^\Z \in \Prob(X)$ and let $\sigma \colon X \to X$ be the bilateral shift on $X$, where 
$$\sigma \om =  \{\om_{n+1}\}_{n\in\Z} \quad \text{for} \ \om = \{\om_n\}_{n\in\Z} \in X  .$$
Define the (invertible) skew-product dynamics
$$
f \colon X\times \T^d\to X \times \T^d, \quad f(\omega,\theta)=(\sigma \omega,\theta+\omega_0) \, .
$$
Let $m$ be the Haar measure on the torus $\T^d$. Then the measure $\mu^\Z\times m$ is $f$-invariant and we call the measure preserving dynamical system $(X\times \T^d, f, \mu^\Z\times m)$ mixed random-quasiperiodic. This system is ergodic if and only if $\hat{\mu} (k) \neq 1$ for all $k \in \Z^d \setminus \{0\}$, where
$$ \hat\mu(k) := \int_{\Sigma} e^{2\pi i \langle k, \theta\rangle}\, d\mu(\theta) $$
are the Fourier coefficients of the measure $\mu$.

Let $\Hscr_\alpha(X\times \T^d)$ be the space of (uniformly) in each variable $\alpha$-H\"older continuous observables (see Definition~\ref{Holder space} for its formal meaning). The goal is then to establish LDT estimates and a CLT for the corresponding observed mixed random-quasiperiodic system  under some appropriate, general condition on the measure $\mu$.
We will associate to this deterministic dynamical system various stochastic dynamical systems, which will be shown to be strongly mixing on appropriate subspaces of $\Hscr_\alpha(X\times \T^d)$. 

For the first and simplest SDS, consider the Markov chain on $\T^d$
$$
\theta \to \theta+\omega_0 \to \theta+\omega_0+\omega_1\to \cdots
$$
where the frequencies $\om_0, \om_1, \ldots$ are i.i.d. according to the probability measure $\mu$.
The corresponding Markov kernel $K \colon \T^d \to \Prob (\T^d)$ is
$$
K_\theta= \int_{\T^d} \delta_{\theta+\omega_0}d\mu(\omega_0)
$$
and the corresponding Markov operator is given by
$$
\Qop=\Qop_\mu \colon L^\infty (\T^d) \to L^\infty(\T^d), \quad \Qop\varphi(\theta)=\int_{\T^d} \varphi(\theta+\omega_0)d\mu(\omega_0) .
$$

We introduce a general condition on the measure $\mu$ that guarantees the strong mixing property of the Markov operator $\Qop_\mu$ on the Banach algebra 
$C^\alpha(\T^d)$ of $\alpha$-H\"older continuous functions on the torus.

\begin{definition}\label{DC mixing intro}
A measure $\mu\in \Prob(\T^d)$ is said to satisfy a mixing Diophantine condition (mixing DC) if
	$$
	\abs{\hat{\mu}(k)}\leq 1-\frac{\gamma}{\abs{k}^\tau}, \,\forall\,k \in \Z^d\backslash\{0\},
	$$
	for some parameters $\gamma, \tau>0$.
\end{definition}

If $\mu \ll m$, it is not hard to see that there is $\sigma \in (0, 1)$ such that $\abs{\hat\mu(k)} \le \sigma < 1$ for all $k \neq 0$, so evidently $\mu$ satisfies a mixing DC.
If, on the other hand, $\mu$ is a Dirac measure (so that the corresponding SDS is just a deterministic torus translation) then the Fourier coefficients of $\mu$ have modulus $1$ and evidently a mixing DC is not satisfied. Moreover, the corresponding Markov/Koopman operator cannot be strongly mixing, but only weakly  mixing (in the sense of Ces\`aro averages) with power rate, provided that the translation frequency satisfies a Diophantine condition (see~\cite{KLM}).  Beyond these two extreme examples, it turns out that most measures $\mu$ satisfy a mixing DC, that is, the condition is prevalent (see Section~\ref{mixing}). The following result states that the Markov operator is then strongly mixing on $C^\alpha(\T^d)$, so an LDT estimate holds for such observables by Theorem~\ref{abstract ldt intro}.

\begin{theorem}\label{LDT mixed intro}
	If $\mu$ is mixing DC with parameters $\gamma,\tau>0$, then $\Qop$ is strongly mixing with power rate on any space of H\"older continuous functions $C^\alpha(\T^d)$. More precisely, there are $C<\infty$ and $p>0$ such that
	$$
	\bnorm{\Qop^n\varphi -\int \varphi dm}_{\infty} \leq C  \frac{1}{n^p} \, \norm{\varphi}_\alpha , \quad \forall\, \varphi\in C^\alpha(\T^d), \forall n \ge 1.
	$$
	In fact, $p$ can be chosen to be $\frac{\alpha}{\tau} -\iota$, for any $\iota > 0$, in which case $C$ will depend on $\iota$. 
	
Moreover, an effective LDT estimate holds for the Markov chain on the torus
$
\theta \to \theta+\omega_0 \to \theta+\omega_0+\omega_1\to \cdots
$ starting from any point $\theta \in \T^d$ and with any observable $\varphi \in C^\alpha(\T^d)$. More precisely,  $\forall\, \epsilon>0$ there is $n (\ep) \in \N$ such that $\forall n \ge n (\ep)$ and $\forall\, \theta\in \T^d$  we have
	$$
	\mu^\N\left\{\abs{\frac{1}{n}[\varphi(\theta)+\cdots+\varphi(\theta+\om_0 + \cdots+\omega_{n-1})]-\int \varphi dm}>\epsilon\right\}< e^{-c(\epsilon)n}
	$$
where  $c(\epsilon)$ and $n (\ep) $ depend explicitly and uniformly on the data, namely on $\ep, \norm{\varphi}_\alpha, \ga, \tau$.
\end{theorem}

We note that if the measure $\mu$ is only mixing, which is equivalent to $\abs{\hat \mu (k)} < 1$ for all $k\neq 0$ (see Theorem~\ref{mixing characterizations}), 
then $\Qop^n\varphi \to \int \varphi dm$ uniformly for any $\varphi \in C^0 (\T^d)$. Thus a {\em non effective} LDT estimate holds for the Markov chain 
$
\theta \to \theta+\omega_0 \to \theta+\omega_0+\omega_1\to \cdots
$ starting from any point $\theta \in \T^d$ and for any continuous observable. The same would hold (modulo the necessary adaptations) if instead of the torus $\T^d$ we considered any compact, Hausdorff, abelian group. This is a particular version of the main result in a recent work by G. Monakov~\cite{Grisha1} on large deviations type estimates for non-stationary random walks on compact, metrizable, abelian groups.

\medskip

If $\alpha>\tau$ so that the mixing rate of the Markov operator is $r_n = C \frac{1}{n^p}$ with $p>1$, using an abstract CLT for Markov chains by Gordin and Liv\v{s}ic (see~\cite{Gordin-L}) we obtain the following.

\begin{theorem}\label{CLT1 intro}
Assume that $\mu \in \Prob (\T^d)$ satisfies a mixing DC with parameters $\gamma,\tau > 0$ and let $\alpha>\tau$. Then for every $\varphi \in C^\alpha(\T^d)$ nonzero with zero mean, there exists $\sigma = \sigma(\varphi)>0$ such that
	$$
	\frac{S_n\varphi}{\sigma \, \sqrt n} \stackrel{d}{\longrightarrow} \mathcal{N}(0,1) \, .
	$$
More precisely, for Lebesgue a.e. $\theta \in \T^d$ and for all $\lambda \in \R$ we have that 
$$\lim_{n\to\infty} \mu^\N \left\{ \frac{\varphi (\theta) +  \cdots + \varphi (\theta + \om_0 + \cdots +\om_{n-1})}{\sigma \sqrt{n}}  \le \la \right\} = 
\int_{- \infty}^\la e^{- \frac{x^2}{2}} \, \frac{d x}{\sqrt{2 \pi}} \, .$$
\end{theorem}

\smallskip

If the measure $\mu$ is {\em discrete}, then this CLT was already obtained, via completely different methods, by B. Borda, see~\cite[Theorem 4]{Borda}. His formulation of the assumption on the measure $\mu$ is different from ours, but they can be easily shown to be equivalent in the discrete measure setting.  
We note, moreover, that there is a vast literature concerning various types of limit theorems for toral translations, see~\cite{DF-CLT} and references therein.

\begin{remark}
A more recent and sharper abstract CLT due to Maxwell and Woodroofe~\cite{CLT-MW} allows us to derive the convergence to the normal distribution 
$\frac{S_n\varphi}{\sqrt n} \stackrel{d}{\longrightarrow} \mathcal{N}(0,\sigma^2 (\varphi))$ under a weaker assumption regarding the regularity of the observable $\varphi$ relative to the mixing of the measure $\mu$. More precisely, we would only need that  $\alpha > \frac \tau 2$ instead of $\alpha > \tau$. However, we are not able to establish a criterion for the positivity of the standard deviation $\sigma (\varphi)$ in this more general setting.
\end{remark}

\medskip

Observables depending only on one variable are of course quite particular. In Section~\ref{randtransl} we consider an extension of the Markov kernel introduced above, which we prove to be strongly mixing. This leads to LDT estimates (and a CLT) for observables in  $\Hscr_\alpha(X\times \T^d)$ depending only on the past (i.e. on the negatively indexed coordinates). Finally, through a hyperbolic dynamics argument, namely a holonomy reduction (that requires slightly more regularity), we relate 
observables that depend on both past and future to ones that depend only on the past, and transfer the statistical properties satisfied by the latter to the former. Our results are thus as follows.

\begin{theorem}\label{ds gen ldt intro}
	Assume that $\mu \in \Prob (\T^d)$ satisfies a mixing DC with parameters $\gamma, \tau > 0$ and that $\varphi \in \Hscr_\alpha(X\times \T^d)$ with $\alpha > 0$. Then for all $\epsilon >0$  there are $c (\ep) > 0$ and $n (\ep) \in \N$ such that for all $n \ge n (\ep)$ we have
	$$
	\mu^{\Z}\times m\left\{ \abs{\frac{1}{n}S_n\varphi -\int \varphi \, d\mu^\Z \times m}>\epsilon\right\}<e^{-c(\epsilon)n}
	$$
	where $c(\epsilon) = \cldt \, \epsilon^{2+\frac{1}{p}}$ and $n (\ep) = \nldt \, \ep^{- \frac{1}{p}}$ for constants $\cldt > 0$, $\nldt \in \N$ which depend explicitly and uniformly on the data (namely on $\norm{\varphi}_\alpha, \ga, \tau$) and a power $p > 0$ which can be chosen arbitrarily close to $\frac{\alpha}{2 \tau}$.  
\end{theorem}

In order to state the CLT we need the following concept.
\begin{definition}\label{coboundary def}	
We say that a continuous observable $\varphi \colon X \times \T^d \to \R$ is a {\em coboundary} relative to the dynamical system $(X \times \T^d, f, \mu^\Z \times m)$ if there exists a continuous function $\eta \colon X \times \T^d \to \R$ satisfying the cohomological equation
$$\varphi = \eta - \eta \circ f \qquad \mu^\Z \times m-\text{a.e.} $$
\end{definition}	

If $\varphi$ is a coboundary, then obviously it has zero mean, that is, $ \int \varphi \, d ( \mu^\Z \times m ) = 0$; moreover, 
$\displaystyle \frac{S_n\varphi}{\sqrt n} = \frac{\eta - \eta \circ f^n}{\sqrt n} \to 0$
uniformly, so a standard CLT cannot hold. Excluding this case, we obtain the following.

\begin{theorem}\label{CLT introduction}
	Assume that $\mu \in \Prob (\T^d)$ satisfies a mixing DC with parameters $\gamma, \tau > 0$ and let $\alpha>2 \tau$. Given any observable  $\varphi \in \Hscr_\alpha(X\times \T^d)$ with zero mean, if $\varphi$ is not a coboundary then there exists $\sigma = \sigma(\varphi)>0$ such that
	$$
	\frac{S_n\varphi}{\sigma \, \sqrt n} \stackrel{d}{\longrightarrow} \mathcal{N}(0,1) \, .
	$$
\end{theorem}

\medskip

The abstract LDT estimate in Theorem~\ref{abstract ldt intro} is applicable to many other types of  dynamical systems. More precisely, one can recover such previously established results for: uniformly expanding maps on a compact, connected Riemannian manifold or for piecewise expanding maps of an interval (see~\cite{Anselmo-dissertation}); certain types of linear cocycles, namely irreducible, locally constant linear cocycles over Bernoulli and Markov shifts (see~\cite[Chapter 5]{DK-book}) and for fiber-bunched cocycles (or partially hyperbolic projective cocycles over uniformly hyperbolic maps, see~\cite{DKP}). Moreover, we will apply this result in future works to linear cocycles over mixed random-quasiperiodic base dynamics (or random compositions of linear cocycles over a toral translation), which was the initial motivation for the current work. 
Furthermore, we believe that this abstract result may be adapted and used to derive effective LDT estimates for many other types of skew-product dynamical systems.

\medskip

The rest of the paper is organized as follows. In Section~\ref{aldts} we describe the relevant concepts regarding stochastic dynamical systems, we state and prove the abstract LDT estimate in Theorem~\ref{abstract ldt intro}, we recall the abstract CLT of Maxwell and Woodroofe and derive a version thereof to be used in the sequel. In Section~\ref{mixing} we introduce mixing measures and establish the strong mixing of the Markov operator corresponding to such measures. In Section~\ref{randtransl} we study the statistical properties of mixed random-quasiperiodic dynamical systems via extensions and holonomy reduction, thus establishing the other results formulated above.

 
\section{Statistical properties for Markov processes}
\label{aldts}
We begin by recalling some basic concepts.

\begin{definition}
A stochastic dynamical system (SDS) is any continuous map
$K \colon M\to \Prob(M)$ on a Polish metric space $M$ where $\Prob(M)$ denotes the space of Borel probability measures on $M$ endowed with the weak  topology  (see~\cite[Section 2.1]{CDK-paper1} for more details).
\end{definition}

Let $\smu \in \Prob (M)$ be a $K$-stationary measure, that is, a measure such that $\smu = K\ast \smu := \int K_x\, d\smu(x)$. An SDS $K$ induces the Markov operator
$\Qop=\Qop_K \colon L^\infty(M)\to L^\infty(M)$ defined by
$$ (\Qop \varphi)(x):= \int_M \varphi(y)\, dK_x(y) .$$

Note that $\Qop$ is a positive operator with $\Qop \ind = \ind$, hence it is bounded with norm one on $L^\infty(M)$.

\begin{definition}\label{marsys}
A Markov system is a tuple $(M,K,\smu,\Escr)$  where
\begin{enumerate}
	\item $M$ is a Polish metric space,
	\item $K \colon M\to \Prob(M)$ is an SDS,
	\item  $\smu\in\Prob(M)$ is a $K$-stationary measure,
	\item  $\Escr=\left(\Escr,\norm{\cdot}_\Escr\right)$ is a Banach subspace
	of $L^\infty(M)$ such that the inclusion $\Escr  \xhookrightarrow{}  L^\infty(M)$ and the action of $\Qop$ on $\Escr$ are both continuous. In other words there is $M < \infty$ such that
	$\norm{\varphi}_\infty\le  \norm{\varphi}_\Escr$ and
	$\norm{\Qop\varphi}_\Escr\leq M \, \norm{\varphi}_\Escr$,
	for all $\varphi\in \Escr$. 
\end{enumerate}
\end{definition}

\begin{remark}
When $\Escr=C_b (M)$ (the space of bounded, continuous functions on $M$), condition (4) follows from (1)-(3).
\end{remark}

\begin{definition}
\label{def decaying rate}
We call decaying rate  any decreasing sequence $r = \{r_n\}_{n \in \N}$ of positive real numbers such that $r_n \to 0$ as $n\to\infty$. 
\end{definition}

Examples of decaying rates are:  $r_n = \exp(-c \,n)$ with $c >0$ (exponential); $r_n = \exp(-n^b)$ with
$0<b<1$ (sub-exponential); and $r_n = C \, n^{-p}$ with $C < \infty$ and $p>0$ (power or polynomial).

\begin{definition}
\label{strong mixing}
Let $r$ be any decaying rate. We say that $(M, K,\smu,\Escr)$ is strongly mixing with mixing rate $r$  if
 for all $\varphi\in\Escr$ and  $n\in\N$,
\begin{equation}\label{mixing condition}
 \norm{\Qop^n \varphi-\int_M \varphi\, d\smu}_\infty \le r_n \,  \,\norm{\varphi}_\Escr \, . 
 \end{equation}
\end{definition}


%



On the product space  $X^+=M^\N$ consider the sequence of random variables $\{Z_n \colon X^+\to M\}_{n\in\N}$,
 $ Z_n(x):= x_n$ where $x=\{x_n\}_{n\in \N}\in X^+$. By Kolmogorov's extension theorem, given $\pi\in \Prob(M)$  there exists a unique
 probability measure $\Pp_\pi$ on $X^+$ for which $\{Z_n\}_{n\in \N}$ is a  Markov process  with transition probability kernel $K$ and initial probability distribution $\pi$, i.e.,
 such that for every Borel set $A\subset M$ and any $n\geq 1$,
 \begin{enumerate}
 	\item [(a)] $\Pp_\pi[ \, Z_n\in A \,\vert \,  Z_0,Z_1,\ldots, Z_{n-1} ] = K_{Z_{n-1}}(A)$,
 	\item [(b)] $\Pp_\pi[ Z_0 \in A]=\pi(A)$.
 \end{enumerate}
When $\pi=\delta_x$ is a Dirac measure we write $\Pp_x$ instead of $\Pp_{\delta_x}$. When $\pi=\smu$ the probability measure $\Pp=\Pp_\smu$  makes
$\{Z_n\}_n$ a stationary process. This measure is preserved by the one sided
shift $\sigma \colon X^+\to X^+$. We have 
$$ \Pp_\smu (B)=\int_M \Pp_x(B)\, d\smu(x)\quad \text{ and } \quad
\EE_\smu[\varphi]=\int_M \EE_x[ \varphi ]\, d\smu(x)  $$
for any Borel set $B\subset X^+$ and any bounded measurable function
$\varphi \colon X^+\to\R$, where $\EE_x$ stands for the expected value w.r.t. $\Pp_x$ while $\EE_\smu$ denotes the  expected value w.r.t. $\Pp_\smu$.


\medskip

We are ready to state and prove an abstract
large deviations type theorem for strongly mixing Markov processes.

Let $\{Z_n\}_{n\geq 0}$ be the $K$-Markov chain $Z_n \colon X^+ \to M$, $Z_n(x)=x_n$. For an observable $\varphi \colon M \to \R$ and an index $j \ge 0$ let $\varphi_j := \varphi \circ Z_j \colon X^+ \to \R$ and denote by 
$$
S_n \varphi := \varphi_0 + \cdots + \varphi_{n-1} =  \varphi(Z_0)+\cdots + \varphi(Z_{n-1})
$$
the corresponding ``stochastic'' Birkhoff sums. Note that for $x = \{x_n\}_{n\ge 0} \in X^+$ and $n \in \N$,
$$S_n \varphi (x) =  \varphi (x_0) + \cdots +  \varphi (x_{n-1}) \, .$$

\begin{theorem}\label{abstractldt}
	Let $(M, K,\smu, \Escr)$ be a strongly mixing Markov system with mixing rate $r = \{r_n\}_{n\in\N}$. Then for all $\varphi\in \Escr$ and $\epsilon >0$  there are $c (\ep) > 0$ and $n (\ep) \in \N$ such that for all $n \ge n (\ep)$ and $x_0 \in M$   we have
	$$
	\Pp_{x_0}\left\{\abs{\frac{1}{n}S_n\varphi-\int_M \varphi d\smu}>\epsilon\right\}\leq 8 e^{-c(\epsilon)n} \, .
	$$
	
The parameters $c(\epsilon)$ and $n (\ep)$ depend explicitly and uniformly on the input data, namely on $\ep, \norm{\varphi}_\Escr$ and $r$.  
More precisely,  $n (\ep)$ is the first integer $n_0$ such that $r_{n_0} \le \frac{\ep}{4 \norm{\varphi}_\Escr}$, while 
$c (\ep) = \frac{1}{8 \, \norm{\varphi}^2_\Escr} \, \frac{\ep^2}{n (\ep)}$.
\end{theorem}

\begin{proof}
Fix $x_0 \in M$ and $\varphi \in \Escr$ and let $L := \norm{\varphi}_\Escr$. Without loss of generality, we may assume that $\int_M\varphi d\smu=0$, otherwise we consider $\varphi-\int_M\varphi d\smu$. Moreover, replacing $\varphi$ by $- \varphi$, it is enough to estimate $\Pp_{x_0}\{S_n \varphi \geq n\epsilon\}$. Using Bernstein's trick, for any $t>0$ we have
	$$
	\Pp_{x_0}\{S_n \varphi \geq n\epsilon\}=\Pp_{x_0}\{e^{t S_n \varphi} \geq e^{tn\epsilon}\}\leq e^{-tn\epsilon} \, \EE_{x_0}(e^{tS_n\varphi}) ,
	$$
thus we need to estimate the exponential moments $\EE_{x_0}(e^{tS_n\varphi})$. This will be achieved  by relating them to powers of the Markov operator $\Qop$ evaluated at a suitably chosen observable, as shown in the next lemma.

\begin{lemma}\label{the lemma} 
		Let $\varphi\in \Escr, \norm{\varphi}_\Escr =: L < \infty$. Let $n\geq n_0$ be two integers and denote by $m:=\lfloor \frac{n}{n_0}\rfloor$. Then for all $x_0\in M$ and for all $t>0$,
		$$
		\EE_{x_0}(e^{tS_n\varphi})\leq e^{2tn_0L}\norm{Q^{n_0}(e^{tn_0\varphi})}_\infty^{m-1}.
		$$
	\end{lemma}
	\begin{proof}[Proof of the lemma] Write $n = m \, n_0 + r$, with $0\le r < n_0$. Fix $t>0$ and let $f := e^{t \varphi} \colon M \to \R$, so $0 < f \le e^{t L}$. Then $\forall x = \{x_n\}_{n\ge 0} \in X^+$ we have 
		\begin{align*}
		e^{t S_n\varphi (x)}  = \prod_{j=0}^{n-1} e^{t \varphi (x_j)}   = & \prod_{j=0}^{n-1}f(x_j) \\
		                  =  & f(x_0)\cdot f(x_{n_0})\cdots f(x_{(m-1)n_0}) \cdot \\
			&  f(x_1)\cdot f(x_{n_0+1})\cdots f(x_{(m-1)n_0+1}) \cdot \\
			& \vdots\\
			&   f(x_{n_0-1})\cdot f(x_{2n_0-1})\cdots f(x_{mn_0-1}) \cdot \\
			&  f(x_{mn_0})\cdot f(x_{mn_0+1})\cdots f(x_{mn_0+r-1}) \\
			=: & F_0 (x) \cdot F_1 (x) \cdots F_{n_0-1} (x) \cdot F_{n_0} (x)
		\end{align*}
where $F_k \colon X^+ \to \R$, $F_k (x) := f (x_k) \cdot f (x_{n_0 + k}) \cdots f(x_{(m-1)n_0+k})$ for $0 \le k \le n_0-1$ and $F_{n_0} (x) :=  f(x_{mn_0})\cdot f(x_{mn_0+1})\cdots f(x_{mn_0+r-1})	$. 	
		
By H\"older's inequality, 
$$
\EE_{x_0} (e^{t S_n\varphi} ) =  \EE_{x_0}( F_0\cdots F_{n_0-1} \cdot F_{n_0}  ) \le  \prod_{k=0}^{n_0-1}[\EE_{x_0}(F_k^{n_0})]^{\frac{1}{n_0}} \cdot \norm{F_{n_0}}_\infty  \, .
$$
		
Note that $\norm{F_{n_0}}_\infty    \le e^{t r L} \le   e^{tn_0L}$. 
We will show that $$\EE_{x_0}(F_k^{n_0})\leq e^{tn_0L}\norm{\Qop^{n_0}(e^{tn_0\varphi})}_\infty^{m-1} \quad  \forall \,k=0,\cdots, n_0-1$$ which will conclude the proof.
		
Fix $k \in \{0, \ldots, n_0-1\}$ and note that
$$
F_k^{n_0} (x) =e^{tn_0\varphi(x_k)}e^{tn_0\varphi(x_{n_0+k})}\cdots e^{tn_0\varphi(x_{(m-1)n_0+k})} \,  .
$$

To simplify notations, let $G \colon X^+ \to \R$, $G(x) :=F_k^{n_0} (x)$ and $g \colon M \to \R$, $g(a) := e^{tn_0\varphi(a)}$. Then $0 < g \le e^{tn_0 L}$ and
$$
G(x)=g(x_k) \cdot g(x_{n_0+k})\cdots g(x_{(m-1)n_0+k}) , 
$$
which is a function that depends on a finite and {\em sparse} set of coordinates, arranged in an arithmetic progression of length $m$ with distance $n_0$ between consecutive terms. We will show that
	        $$
		\EE_{x_0}(G)\leq e^{tn_0L}\norm{Q^{n_0}g}_\infty^{m-1}
		$$
		where
		$$
		\EE_{x_0}(G)=\int_{X^+} G(x)d\Pp_{x_0}(x)=\int_{X^+} G(x)\prod_{i=n}^{1}dK_{x_{i-1}}(x_i) \, .
		$$
		
		We split the set $I = \{1,2,\ldots, (m-1)n_0+k\}$ of $(m-1)n_0+k$ many indices  into
		$$
		I=\left\{1,\ldots,k\right\}\cup I_1 \cup \cdots \cup I_{m-1}
		$$
		where for $j=1,\ldots, m-1$
		$$
		I_j:=\left\{(j-1)n_0+k+1,\cdots, jn_0+k\right\}
		$$
		is a block of length $n_0$. 
		
Since $G (x)$ does not depend on the variables $x_j$ with $j\notin I$, we have
		\begin{align}
		\EE_{x_0}(G)=&\int_{X^+} G(x)\prod_{i=n}^{1}dK_{x_{i-1}}(x_i) \notag \\
		=&\int g(x_k)\cdots g(x_{(m-2)n_0+k})\left( \int g(x_{(m-1)n_0+k}) \prod_{i\in I_{m-1}}dK_{x_{i-1}}(x_i)\right)  \notag \\
		&\prod_{i\in I\backslash I_{m-1}}dK_{x_{i-1}}(x_i) \notag \\
		=&\int g(x_k)\cdots g(x_{(m-2)n_0+k}) \left(\Qop^{n_0}g(x_{(m-2)n_0+k}) \right)  \prod_{i\in I\backslash I_{m-1}}dK_{x_{i-1}}(x_i)  \notag \\
       \leq & \norm{\Qop^{n_0}g}_\infty \int g(x_k)\cdots g(x_{(m-2)n_0+k})\prod_{i\in I\backslash I_{m-1}}dK_{x_{i-1}}(x_i) \, . \label{ineq lemma}
	\end{align}

Applying the same argument $m-1$ times, we obtain
\begin{align*}
\EE_{x_0}(G)&\leq\norm{\Qop^{n_0}g}_\infty^{m-1}\cdot \int g(x_k)dK_{x_{k-1}}(x_k)\cdots dK_{x_0}(x_1)\\
&\leq \norm{\Qop^{n_0}g}_\infty^{m-1} e^{tn_0L} ,
\end{align*}
which completes the proof of the lemma.
\end{proof}

We return to the proof of the theorem. Using the strong mixing assumption, for all $n_0 \in \N$ and $\varphi \in \Escr$ we have that
$$
\norm{\Qop^{n_0}\varphi}_\infty \le r_{n_0} \,  \norm{\varphi}_\Escr \,  .
$$ 
	
By Lemma~\ref{the lemma}, for all $n\geq n_0$, 
	$$
	\EE_{x_0}(e^{tS_n\varphi}) \le e^{2tn_0L} \, \norm{\Qop^{n_0}(e^{tn_0\varphi})}_\infty^{\frac{n}{n_0}-1}.
	$$

However, $\varphi\in \Escr$ does not necessarily imply that $e^{ t n_0\varphi}\in \Escr$; even if this were the case, its mean would not be zero, so the strong mixing condition cannot be directly applied to the observable $e^{tn_0\varphi}$.  
	
The following inequality holds for all $y \in \R$:
	$$
	e^y\leq 1+y+\frac{y^2}{2}\, e^{\abs{y}} .
	$$
	
Hence we can write
	$$
	e^y= 1+y+\frac{y^2}{2}\, \psi(y)
	$$
	where the function $\psi$ satisfies the bound $\abs{\psi(y)}\le e^{\abs{y}}$. Then
	$$
	e^{tn_0\varphi}=1+tn_0\varphi+\frac{1}{2}t^2n_0^2\varphi^2\psi(tn_0\varphi)
	$$
	where $\norm{\psi(tn_0\varphi)}_\infty \le  e^{tn_0\norm{\varphi}_\infty} \le  e^{tn_0L} \leq 2$ if  $t n_0 L \le \log 2$. 
	
Applying the Markov operator we then have:
\begin{align*}
\Qop^{n_0}(e^{tn_0\varphi}) & = 1 + t n_0 \, \Qop^{n_0} (\varphi) + \frac{1}{2}t^2n_0^2 \, \Qop^{n_0}(\varphi^2\psi(tn_0\varphi)) \, . 
\end{align*}

Therefore,
	$$
	\norm{\Qop^{n_0}(e^{tn_0\varphi})}_\infty\leq 1+t n_0 r_{n_0} L + t^2 n_0^2 L^2 \le 1+2t^2 n_0^2  \, L^2 
	$$
provided that $r_{n_0} \le t n_0 L \le \log 2$.

Using the inequality $(1+y)^{\frac{1}{y}}\leq e$ for $y>0$ we get
	$$
	\norm{\Qop^{n_0}(e^{tn_0\varphi})}_\infty^{\frac{n}{n_0}}\leq(1+2t^2n_0^2L^2)^{\frac{1}{2t^2n_0^2L^2 }\cdot 2t^2n_0^2L^2  \cdot \frac{n}{n_0}}\leq e^{2t^2n_0L^2  \, n}.
	$$
	
Combining this with the estimate given by Lemma~\ref{the lemma} and recalling that $e^{tn_0L} \leq 2$, we get
	$$
	\EE_{x_0}(e^{tS_n\varphi})\leq e^{2tn_0L} \, e^{2t^2n_0L^2  \, n} \le 4 e^{2t^2n_0L^2  \, n} \, .
	$$
	
Fix $\ep > 0$. Using  Bernstein's trick, for all $t > 0$ we have
$$
\Pp_{x_0}\{S_n \varphi \geq n\epsilon\}  \le e^{-t n \epsilon} \, \EE_{x_0}(e^{tS_n\varphi}) \le 4 e^{-t n \epsilon} e^{2t^2 n_0 L^2  n} = 4 e^{- n (t \epsilon - 2 t^2 n_0 L^2 )} \, .
$$
	
Then	
$$\Pp_{x_0}\{S_n \varphi \geq n\epsilon\} \le 4 e^{- c (\ep) n}$$
where 
$ c (\ep) = \max_{t>0} (t \epsilon - 2 t^2 n_0 L^2 ) = \frac{1}{8 L^2} \, \frac{\ep^2}{n_0}$
is attained when $t = \frac{\ep}{4 n_0 L^2}$.

This choice of the parameter $t$ is compatible with the constraint $t n_0 L \le  \log 2$, since this is equivalent to $t n_0 L = \frac{\ep}{4 L} \le \log 2$, which of course must hold (the size $\ep$ of the deviation cannot exceed $2 \norm{\varphi}_\infty \le 2 L$). The other constraint,  $r_{n_0} \le t n_0 L$, is equivalent to $r_{n_0} \le \frac{\ep}{4 L}$. Therefore we choose $n_0 = n (\ep)$ to be the first  integer such that $r_{n_0} \le \frac{\ep}{4 L}$, which completes to proof  of the theorem.
\end{proof}

\begin{remark}\label{rem2.3}
Theorem~\ref{abstractldt} has an equivalent but more abstract probabilistic version. Let $(\Omega, \mathcal{P})$ be any probability space and let $(M, K, \smu, \Escr)$ be a Markov system. Given a $K$-Markov chain $\{\xi_n:\Omega \to M\}_{n\geq 0}$ with $\xi_0 = x_0$, where $x_0 \in M$,
define $\xi: \Omega \to M^{\N}$ by $\xi(\omega):= \{\xi_n(\omega)\}_{n\geq 0}$. A direct computation shows that $\xi_\ast \mathcal{P}=\Pp_{x_0}$, the Markov measure on $M^\N$ with initial distribution $\delta_{x_0}$. Thus under the same assumptions in Theorem~\ref{abstractldt}, we have
$$
\mathcal{P}\left[\omega \in \Omega  :\lvert \, \frac{1}{n} \sum_{j=0}^{n-1} \varphi(\xi_j(\omega)) - \EE_\smu (\varphi)\, \rvert >\epsilon
 \, \right] \leq  e^{- c(\ep) n} \, .
$$
\end{remark}

\begin{remark}\label{strongldt} In the mixing condition~\eqref{mixing condition}, $\norm{ \cdot }_\infty$ refers to the {\em uniform} norm inherited from the space $L^\infty (M)$ of bounded, measurable functions, that is, $\norm{\varphi}_\infty := \sup_{x\in M} \abs{\varphi (x)}$. 

If we use instead the {\em essential} supremum norm relative to the stationary measure $\smu$, then an LDT estimate still holds, but only with the Markov measure  with initial distribution $\smu$ (the stationary measure), that is,
$$
\Pp_{\smu}\left\{\abs{\frac{1}{n}S_n\varphi-\int_M \varphi d\smu}>\epsilon\right\}\leq 8 e^{-c(\epsilon)n} \, .
$$

Besides having to replace $\Pp_{x_0} ,\EE_{x_0}$ by $\Pp_\smu,\EE_\smu$,  the argument is essentially the same, with one exception. Estimate~\eqref{ineq lemma} is a consequence of the bound $\Qop^{n_0}g (x_{(m-2)n_0+k}) \le \norm{\Qop^{n_0}g}_\infty$, which in this setting does not always hold. 
However, since the bound 
$$\Qop^{n_0} g (a) \le \norm{\Qop^{n_0} g}_\infty$$
holds for all points $a \in M$ except for a set $\B_{n_0}$ of $\smu$-measure zero, and since $\smu$ is $K$-stationary, hence the measure $\Pp_\smu$ on $X^+$ is $\sigma$-invariant, it follows that
\begin{align*}
\Pp_\smu \left\{ x \in X^+ \colon \exists j \in \N, x_j \in \B_{n_0}  \right\} & = \Pp_\smu \left\{ x \in X^+ \colon  \exists j \in \N, \sigma^j x \in C [\B_{n_0}]  \right\} \\
& = \Pp_\smu \left( \bigcup_{j\ge 0} \sigma^{- j} C [\B_{n_0}]  \right) \\
&  \le \sum_{j=0}^\infty \Pp_\smu \left(  \sigma^{- j} C [\B_{n_0}]  \right) = \sum_{j=0}^\infty \Pp_\smu \left( C [\B_{n_0}]  \right) \\
& = \sum_{j=0}^\infty \smu \left( \B_{n_0}  \right) = 0 \, .
\end{align*}
Therefore $\Qop^{n_0}g (x_j) \le \norm{\Qop^{n_0}g}_\infty$ for all indices $j\in\N$ holds for a $\Pp_\smu$-full measure set of $x\in X^+$, which implies the claim.

An interesting question is whether a slightly weaker version of the LDT estimate could be obtained if we replaced the  $L^\infty$-norm in the mixing condition~\eqref{mixing condition} by an $L^p$-norm. This will be considered in a separate project.
\end{remark}

\begin{remark}\label{non effective LDT remark}
If we weaken the mixing condition~\eqref{mixing condition} by requiring only that for an observable $\varphi\in L^\infty(M)$, 
 $\Qop^n \varphi \to \int_M \varphi\, d\smu$ uniformly, then we obtain the following {\em non effective} LDT estimate. For any $\ep > 0$ there are $n (\ep) \in \N$ and $c (\ep) > 0$ such that for all $x_0 \in M$ we have
$$
\Pp_{x_0}\left\{\abs{\frac{1}{n}S_n\varphi-\int_M \varphi d\smu}>\epsilon\right\}\leq 8 e^{-c(\epsilon)n} \, .
$$

The argument is the same. Assuming again that $\int \varphi d \smu = 0$, so $\norm{\Qop^{n} (\varphi)}_\infty \to 0$, the threshold $n (\ep)$ will be the first integer $n_0$ such that $\norm{\Qop^{n_0} (\varphi)}_\infty \le \frac{\ep}{4 L}$, where $L := \norm{\varphi}_{\infty}$, while $c (\ep)$ has the same expression.

Note that the parameters $n (\ep)$ and $c (\ep)$ depend in a uniform but not explicit way on the observable $\varphi$ (in other words, they do not change much as we vary $\varphi \in L^\infty (M)$, but the threshold for the limiting behavior cannot be determined from the input data).   
\end{remark}

We now recall an abstract central limit theorem of Gordin and Liv\v{s}ic (see~\cite{Gordin-L} and~\cite{Gordin2}).

\begin{theorem}[Gordin-Liv\v{s}ic]\label{abstractCLT}
	Let $(M, K, \smu)$ be an ergodic Markov system, let $\varphi \in L^2(\smu)$ with $\int \varphi d\smu=0$ and assume that
	$$
	\sum_{n=0}^{\infty}\norm{\Qop^n \varphi}_2 <\infty.
	$$
	
Denoting $\psi :=\sum_{n=0}^{\infty}\Qop^n \varphi$, we have that $\psi \in L^2(\smu)$ and $\varphi =\psi - \Qop \psi$. 
If $\sigma^2 (\varphi) :=\norm{\psi}_2^2- \norm{\Qop\psi}_2^2>0$ then the following CLT holds:
	$$
	\frac{S_n\varphi}{\sigma (\varphi)  \, \sqrt n} \stackrel{d}{\longrightarrow} \mathcal{N}(0,1) \, .
	$$
\end{theorem}

Recall that a Markov system $(M, K, \smu)$ is ergodic if the measure $\smu$ is an extremal point in the convex space of $K$-stationary probability measures on $M$. This is equivalent to the ergodicity of the shift map on the product space $X^+$ relative to the Markov measure $\Pp = \Pp_\smu$.
Evidently, if $K$ admits a unique stationary measure, then the corresponding Markov system is ergodic.

As a consequence of the above result we obtain the following.

\begin{proposition}\label{corclt}
	Let $(M, K, \smu, \Escr)$ be a strongly mixing Markov system (relative to the uniform norm) with mixing rate $r_n= C \frac{1}{n^p}$ with $p>1$, where $\Escr$ is a dense subset of $C_b (M)$. 
	
Assume that for any open set $U\subset M$ with $\smu (U) > 0$ there exists $\phi\in \Escr$ such that $0\leq \phi \leq \ind_U$ and $\displaystyle \int_M \phi d\smu>0$. 
For any observable $\varphi\in \Escr$, if $\varphi$ is not $\smu$-a.e. constant then~Theorem~\ref{abstractCLT}  is applicable with $\sigma^2 (\varphi) > 0$ and the CLT holds.
\end{proposition}

\begin{proof} The strong mixing condition and the density of $\Escr$ in $C_b (M)$ imply the uniqueness of the $K$-stationary measure, which in turn imply the ergodicity of the Markov system. Indeed, if  $\tilde \smu$ is a $K$-stationary measure, then for any $\varphi \in C_b (M)$ we have $\int \Qop^n \varphi \, d \tilde \smu = \int \varphi \, d \tilde \smu$ for all $n \in \N$. 
By strong mixing, for any $\varphi \in \Escr$ we have that 
$\Qop^n \varphi \to \int \varphi d \smu $ uniformly. Integrating with respect to $\tilde \smu$ we conclude that
$\int \varphi \, d \tilde \smu = \int \varphi \, d \smu$ for all $\varphi \in \Escr$, so for all $\varphi \in C_b (M)$, which shows that $\tilde \smu = \smu$.

Let $\varphi\in \Escr$ be a non $\smu$-a.e. constant observable. We may of course assume that $\int \varphi d \smu =0$, otherwise we consider $\varphi - \int \varphi d \smu$. 

Let $\psi :=\sum_{n=0}^{\infty}\Qop^n \varphi$. Since $\varphi \in C_b (M)$, the strong mixing assumption on $\Qop$ implies (via the Weierstrass $M$-test) that $\psi \in C_b (M)$ as well. It remains to show that $\sigma^2 (\varphi) > 0$ which ensures the applicability of Theorem~\ref{abstractCLT}.

Assume by contradiction that $\sigma^2 (\varphi) = \norm{\psi}_2^2-\norm{\Qop\psi}_2^2= 0$.
Then
\begin{align*}
0 &\leq  \int  \left( (\Qop\psi)(x) - \psi(y) \right)^2\, dK_x(y)\, d\smu(x)\\ 
&=  \int  \left\{  ((\Qop \psi)(x))^2+ \psi(y)^2
-2\, \psi(y)\, (\Qop \psi)(x)\right\} \, dK_x(y) \, d\smu(x)\\
&= \int  \left\{  \psi(y)^2
- ((\Qop\psi)(x))^2 \right\} \, dK_x(y)\, d\smu(x)\\
&= \int   \psi(y)^2\, dK_x(y)\, d\smu(x) - \int 
((\Qop\psi)(x))^2   \,  d\smu(x)\\
&= \norm{\psi}_2^2- \norm{\Qop\psi}_2^2= 0 \quad (\text{since } \smu \text{ is } K-\text{stationary}) . 
\end{align*}

Therefore, $\psi (y) = \Qop\psi (x)$ for $\smu$-a.e. $x\in M$ and $K_x$-a.e. $y\in M$.  	
By induction we obtain that for all $n\geq 1$,
\begin{equation*}\label{psiconst}
	\psi(y)=(\Qop^n\psi)(x) \quad\text{for $\smu$-a.e. $x\in M$ and for $K_x^n$-a.e. $y\in M$},
\end{equation*}
which implies that for all $n\ge 1$ and for $\smu$-a.e. $x\in M$, the function $\psi$ is $K_x^n$-a.e. constant. Let us show that in fact $
\psi$ is $\smu$-a.e. constant. 

If $\psi$ is not $\smu$-a.e constant, then there exist two disjoint open subsets $U_1$ and $U_2$ of $M$ such that $\smu (U_1), \smu (U_2) > 0$ and $\psi|_{U_1} < \psi|_{U_2}$. By the assumption, there are two observables $\phi_1, \phi_2 \in \Escr$ such that $0\leq \phi_i \leq \ind_{U_i}$ and $\displaystyle \int \phi_i \, d\smu>0$ for $i = 1, 2$. 
Moreover, for all $x\in M$ and $n\ge1$,
$$
K_x^n (U_i) = (\Qop^n \ind_{U_i})(x)\geq (\Qop^n \phi_i)(x)\to \int \phi_i d\smu>0 ,
$$
where the above convergence as $n\to \infty$ is uniform in $x\in M$.

Thus for a large enough integer $n$ and for all $x\in M$, both sets $U_1$ and $U_2$ have positive $K_x^n$ measure.  However, $\psi|_{U_1} < \psi|_{U_2}$, which contradicts the fact that $\psi$ is $K_x^n$-a.e. constant for $\smu$-a.e. $x\in M$.

We conclude that $\psi$ is $\smu$-a.e constant. Since $\smu$ is $K$-stationary it follows that $\varphi=\psi-\Qop\psi=0$ $\smu$-a.e, which is a contradiction.
\end{proof}

We note that Theorem~\ref{abstractCLT} holds not only for the probability $\Pp = \Pp_\smu$, but also for the probability $\Pp_{x_0}$ corresponding to the Markov chain starting from $\smu$-a.e. point $x_0 \in M$ (see the comments after Definition 1.1 in~\cite{Gordin2}). Then Proposition~\ref{corclt} and all of its consequences, e.g. Theorem~\ref{CLT1 intro}  and Theorem~\ref{CLT introduction}, also hold w.r.t. these measures.


\section{Mixing measures}
\label{mixing}
Let $\mu \in \Prob(\T^d)$. We consider the Markov chain on $\T^d$:
$$
\theta \to \theta+\omega_0 \to \theta+\omega_0+\omega_1\to \cdots
$$
The corresponding Markov kernel $K$ is
$$
K_\theta= \int_{\T^d} \delta_{\theta+\omega_0}d\mu(\omega_0)
$$
which is obviously continuous with respect to the weak$^\ast$ topology.
The corresponding Markov operator is 
$$
\Qop=\Qop_\mu: L^\infty (\T^d) \to L^\infty(\T^d), \quad \Qop\varphi(\theta)=\int_{\T^d} \varphi(\theta+\omega_0)d\mu(\omega_0) \, .
$$
Note that $\Qop$ is bounded on $L^\infty(\T^d,m)$ because the Lebesgue measure $m$ is translation invariant, which also ensures that $m$ is $K$-stationary. Hence $(\T^d, K, m)$ is a Markov system.

The goal of this section is to show that the Markov operator $\Qop$ of the Markov system $(\T^d, K, m)$ is strongly mixing under certain assumptions on $\mu$ and for an appropriate space of observables, which will then allow us to derive an LDT and a CLT in this setting, via the abstract results of the previous section. These statistical properties will  be lifted to other Markov systems in Section~\ref{randtransl}.

Recall that the Fourier coefficients of a measure $\mu\in \Prob(\T^d)$ are defined as
$$
\hat{\mu}(k):= \int_{\T^d} e_k (x) \, d\mu(x),
$$
where $e_k \colon \T^d \to \C$, $e_k(x):= e^{2\pi i \langle k,x\rangle }$ for $k \in \Z^d$ are the characters of the multiplicative group $\T^d$.

\begin{lemma}
	The characters $\{e_k \colon  k\in \Z^d\}$ form a complete basis of eigenvectors for the Markov operator $\Qop \colon  L^2(\T^d) \to L^2(\T^d)$. That is, $\Qop e_k=\hat{\mu}(k) e_k, \forall\, k\in\Z^d$ and if $\varphi=\sum_{k\in \Z^d}\hat{\varphi}(k)e_k$ in $L^2(\T^d,m)$, then
	$$
	\Qop \varphi=\sum_{k\in \Z^d}\hat{\mu}(k)\hat{\varphi}(k)e_k \quad \text{in} \quad L^2(\T^d,m).
	$$
\end{lemma}

\begin{proof}
	By the linearity and boundedness of $\Qop$, it is enough to prove the first equality. For any $\theta\in \T^d$ and any $k\in\Z^d$, we have
	\begin{align*}
		\Qop e_k(\theta)&=\int_{\T^d} e_k(\theta+\omega_0)d\mu(\omega_0)\\
		&=\int_{\T^d} e_k(\theta)e_k(\omega_0)d\mu(\omega_0)\\
		&=e_k(\theta)\int e_k d\mu=e_k(\theta)\hat{\mu}(k).
	\end{align*}
	Thus the result follows.
\end{proof}

\begin{remark}
	It turns out (see~\cite[Theorem 2.3]{CDK-paper1}) that the mixed random-quasiperiodic system $(\Sigma^\Z\times \T^d, f, \mu^\Z \times m)$ is ergodic if and only if $\hat{\mu}(k)\neq 1, \forall\, k\in \Z^d\backslash\{0\}$. Furthermore, this is equivalent to the convergence  
	$$
	\frac{1}{n}\sum_{j=0}^{n-1}\Qop^j\varphi(\theta)\to \int \varphi dm \,\,\text{ as }\,\, n\to \infty 
	$$
for all $\theta \in \T^d$ and $\varphi \in C^0(\T^d)$.	
\end{remark}

Our goal is to provide a criterion for the {\em strong mixing} of the Markov operator $\Qop$ (determined by the measure $\mu$), a property which is strictly stronger than ergodicity.

We begin with the intermediate property of mixing.

\begin{definition}\label{def mixing}
The system	$(\Qop, m)$ is called mixing if $\forall\,\varphi\in C^0(\T^d)$, $$\Qop^n\varphi(\theta)\to \int \varphi dm \,\, \text{ as }\,\, n\to\infty, \quad\forall\, \theta \in \T^d.$$
\end{definition}

It is clear that the mixing of $(\Qop,m)$ implies the ergodicity of the measure preserving dynamical system $(f, \mu^\Z \times m)$. Moreover, as it will be seen below, if it holds, the convergence in Definition~\ref{def mixing} must be uniform in $\theta$.
Furthermore, the mixing property can be characterized as follows (see~\cite[Theorem 2.3]{CDK-paper1} for the analogue characterization of ergodicity).

\begin{theorem}\label{mixing characterizations}
	The following statements are equivalent.
	\begin{enumerate}
		\item $(\Qop_\mu, m)$ is mixing.
		\item $\abs{\hat{\mu}(k)}<1, \forall\, k\in \Z^d\backslash\{0\}$.
		\item $\forall\, k\in \Z^d\backslash\{0\}$, $\forall\,  E\subset \T^d$ with $\mu(E)=1$, $\exists\,\alpha \neq \beta \in 
		E$  so that $\langle k,\alpha-\beta\rangle \notin \Z$.
		\item The semigroup generated by the set $S := \{ \alpha - \beta \colon \alpha, \beta \in \supp (\mu) \}$ is dense in $\T^d$.
	\end{enumerate}
If any of these statements holds, we may also call the measure $\mu$ mixing.	
\end{theorem}

\begin{proof}
	$(1)\Rightarrow(2)$. If there is $k\in \Z^d\backslash\{0\}$ such that $\abs{\hat{\mu}(k)}=1$, then since $\Qop^n e_k=\hat{\mu}(k)^ne_k$ for all $n\in\N$, we have
	$$
	\abs{\Qop^n e_k}=\abs{\hat{\mu}(k)^n e_k}=1 \nrightarrow 0= \abs{\int e_k \, d m}  \,\text{ as }\, n\to\infty.
	$$
	This contradicts the mixing condition. 
	
	$(2)\Rightarrow(1)$. We first establish the convergence in Definition~\ref{def mixing} for trigonometric polynomials, then proceed by approximation. 
	
Let $p=\sum_{\abs{k} \le N} c_k e_k$ be a trigonometric polynomial. Note that $\int p dm=c_0$ and $\hat{\mu}(0)=1$, so we have
	$$
	\Qop^n p-\int pdm=\sum_{0<\abs{k}\leq N}c_k \hat{\mu}(k)^n e_k.
	$$
	Hence
	$$
	\norm{ \Qop^n p-\int pdm}_{\infty} \leq \sum_{0<\abs{k}\leq N} \abs{c_k}\abs{\hat{\mu}(k)}^n \, .
	$$
	Let $\sigma :=\max\left\{\abs{\hat{\mu}(k)} \colon 0<\abs{k}\leq N\right\}<1$. Then
	$$
	\norm{ \Qop^n p-\int p dm}_{\infty} \le \Big(\sum_{0<\abs{k}\leq N} \abs{c_k} \Big) \, \sigma^n \to 0\,\text{ as }\, n\to 0.
	$$

Given any observable $\varphi \in C^0(\T^d)$ and given $\epsilon>0$, by the Weierstrass approximation theorem there exists a trigonometric polynomial $p$ such that $\norm{\varphi-p}_{\infty} <\epsilon$. Moreover, by the previous argument, there is $n(\ep) \in \N$ such that 
$\norm{Q^np -\int p \, d m}_{\infty} <\epsilon$ for all $n \ge n (\ep)$.
	
Writing $\varphi=p+\varphi-p$, we  have 
$$
	\Qop^n\varphi=\Qop^np+\Qop^n(\varphi-p)
	$$
	and
	$$
	\int \varphi dm =\int p dm +\int (\varphi-p)dm.
	$$
	
Then for all $n \ge n (\ep)$,
	\begin{align*}
		\norm{\Qop^n \varphi -\int \varphi dm}_{\infty} & \leq \norm{\Qop^np-\int pdm}_{\infty} +\norm{\varphi-p}_{\infty} +\norm{\Qop^n(\varphi-p)}_{\infty} \\
		& \leq \epsilon+\epsilon+\epsilon=3\epsilon ,
	\end{align*}
which proves the mixing of $(\Qop, m)$ and it also shows that the convergene in Definition~\ref{def mixing} must be uniform.
	
$(2)\Leftrightarrow(3)$. Let  $k\in\Z^d\backslash\{0\}$. Then $\abs{\hat{\mu}(k)}=1$ if and only if $\abs{ \int e^{2\pi i\langle k,\alpha\rangle} \, d \mu (\alpha) } = 1$ which, by the lemma below, is equivalent to $e^{2\pi i\langle k,\alpha\rangle}$ being constant for $\mu$-a.e. $\alpha \in \Sigma$. This holds if and only if there is $E\subset \Sigma$ with $\mu(E)=1$ such that $\forall\, \alpha, \beta \in E$, $e^{2\pi i\langle k,\alpha\rangle}=e^{2\pi i\langle k,\beta\rangle}$. This is equivalent to $e^{2\pi i\langle k,\alpha - \beta \rangle}=1$ for all $\alpha, \beta \in E$, which holds if and only if $\langle k,\alpha - \beta \rangle \in \Z$ for all $\alpha, \beta \in E$, thus establishing the claim.

$(3) \, \Rightarrow\, (4)$.
The closed semigroup $H$ generated by $S$ can be written as  $H=\overline{\cup_{n\geq 1} S^n}$, where $S^n := S + S^{n-1}$.  
By the Poincar\'e recurrence theorem $H$ is also a group. Assuming by contradiction that $H\neq \T^d$,
by Pontryagin's duality for locally compact abelian groups,
there exists a non trivial character $e_k \colon \T^d\to\C$ containing $H$ in its kernel. In particular this implies that
there exists $k\in\Z^d\setminus\{0\}$ such that
$\langle k, \theta \rangle\in\Z$ for all $\theta \in S$, which is a contradiction. 

$(4) \, \Rightarrow\, (3)$.
Assume by contradiction that for some $k\in\Z^d\setminus\{0\}$ and $E\subset\T^d$, with full $\mu$-measure,
we have $\langle k, \beta-\alpha \rangle\in\Z$ or equivalently $e^{2\pi i \langle k, \beta-\alpha \rangle}=1$,
 for all $\alpha, \beta  \in E$. 
Because $E$ is dense in $\supp(\mu)$, this implies by continuity that $e^{2\pi i \langle k, \theta \rangle}=1$
 for all $\theta\in S$. Then $e_k$ is a nontrivial character
of $\T^d$ and $H:=\{ \theta\in\T^d\colon e_k(\theta)=1 \}$
is a proper compact subgroup of $\T^d$.
The assumption implies that $S\subset H$, hence the closed semigroup generated by $S$ is contained in $H$, a contradiction with (4).
\end{proof}

\begin{lemma}\label{lemarg}
	Let $(\Omega,\rho)$ be a probability space. Assume that $f \colon \Omega\to \C$ is Lebesgue integrable. If $\abs{\int_\Omega f d\rho}=\int_\Omega\abs{f}d\rho$, then $\arg f$ is constant $\rho$-a.e. That is, $\exists\, \theta_0\in \R$ such that $f(x)=e^{i\theta_0}\abs{f(x)}$ for $\rho$-a.e. $x\in \Omega$. 
\end{lemma}
\begin{proof}
Let $\theta_0:=\arg(\int_\Omega f d\rho)$, so we can write 
	$$
	\int_\Omega f d\rho=e^{i\theta_0}\abs{\int_\Omega f d\rho}.
	$$
	
	Then
	\begin{align*}
		0=& \abs{\int_\Omega f d\rho}-\int_\Omega \abs{f} d\rho 
		= e^{-i\theta_0}\int_\Omega f d\rho-\int_\Omega \abs{f}d\rho\\
		=& \int_\Omega \left( e^{-i\theta_0}f-\abs{e^{-i\theta_0}f} \right) d\rho
		= \Re \int_\Omega \left(e^{-i\theta_0}f-\abs{e^{-i\theta_0}f} \right) d\rho\\
		=& \int_\Omega \left( \Re (e^{-i\theta_0}f)-\abs{e^{-i\theta_0}f} \right) d\rho \, .
	\end{align*}
	
Since $\Re (e^{-i\theta_0}f) \le \abs{e^{-i\theta_0}f}$, it follows that $\Re (e^{-i\theta_0}f)=\abs{e^{-i\theta_0}f}\geq 0$, $\rho$-a.e. In particular, $\Im(e^{-i\theta_0}f)=0$ $\rho$-a.e. Therefore,
	$$
	e^{-i\theta_0}f=\Re (e^{-i\theta_0}f)=\abs{e^{-i\theta_0}f}=\abs{f} \quad \rho\text{-a.e.}
	$$
	which implies $f=e^{i\theta_0}\abs{f}$ $\rho$-a.e.
\end{proof}

As shown above, the system $(\Qop, m)$ is mixing if and only if $\abs{\hat{\mu}(k)}<1$ for all  $k\in \Z^d\backslash\{0\}$. 
It is not difficult to see that if $\mu \ll m$ then there is $\sigma \in (0, 1)$ such that $\abs{\hat{\mu}(k)} \le \sigma < 1$ for all  $k\in \Z^d\backslash\{0\}$.
This ensures the mixing and in fact the strong mixing with exponential rate of the system $(\Qop, m)$, but it is, of course, a very restrictive assumption.
We introduce a general, in fact generic condition on the measure $\mu$ that allows its Fourier coefficients to approach $1$ but with a controlled speed, and which will imply the strong mixing of $(\Qop, m)$.

\begin{definition}
	We say that $\mu\in \Prob(\T^d)$ satisfies a mixing Diophantine condition (mixing DC) if
	$$
	\abs{\hat{\mu}(k)}\leq 1-\frac{\gamma}{\abs{k}^\tau}, \,\forall\,k \in \Z^d\backslash\{0\},
	$$
	for some $\gamma, \tau>0$. In this case we write  $\mu\in \rm{DC}(\gamma,\tau)$.
\end{definition}

This is inspired by the concept of Diophantine condition (DC) for points on the torus. We say that $\alpha\in \T^d$ satisfies the Diophantine condition DC$(\gamma,\tau)$ if
$$
\inf_{j\in \Z} \abs{\langle k,\alpha \rangle-j}\geq \frac{\gamma}{\abs{k}^\tau},\, \forall\, k\in \Z^d\backslash\{0\}.
$$

It is usually assumed when talking about a DC for points on the torus that $\gamma>0$ and $\tau>d$. This is because when $\tau<d$, the set of points satisfying DC$(\gamma,\tau)$  is empty, when $\tau=d$ it has Lebesgue measure zero on $\T^d$, while when $\tau > d$, the set $\cup_{\ga>0} \rm{DC} (\gamma,\tau)$ has Lebesgue measure one.

However, the set of measures on the torus satisfying a mixing Diophantine condition is always non-empty, for any parameters $\gamma, \tau>0$.

We give some examples of mixing and non mixing DC measures. 

\begin{enumerate}
	\item If $\mu \ll m$, then $\mu$ is mixing DC with any $\tau\geq 0$.
	\item If $\mu=\delta_\alpha$ then $\hat{\mu}(k)=\int e^{2\pi i\langle k,x\rangle}d\delta_\alpha(x)=e^{2\pi i\langle k,\alpha\rangle}$. Thus $\abs{\hat{\mu}(k)} = 1  \, \forall\, k\in \Z$ which implies that $\delta_\alpha$ is not mixing DC.
	\item If $\mu=t\delta_\alpha + (1-t)\delta_\beta$ with $t\in(0,1)$ and $\beta-\alpha\in $ DC, then $\mu$ is mixing DC.
	\item If $\mu\in \Prob(\T^d)$ is finitely supported and $\exists\, \alpha,\beta\in \supp(\mu)$ such that $\beta-\alpha\in$ DC, then $\mu$ is mixing DC.
	\item If $\mu_1\in \Prob(\T^d)$ is mixing DC then for any $t\in (0,1]$ and $\mu_2\in \Prob(\T^d)$, $\mu:=t\mu_1+(1-t)\mu_2$ is mixing DC.
\end{enumerate}

Note that $(5)$ implies that the set of mixing DC measures is prevalent.

\smallskip

The following result shows that the mixing DC of a measure $\mu$ ensures the strong mixing of the corresponding Markov system $(\Qop, m)$.

\begin{proposition}
	If $\mu\in \rm{DC}(\gamma,\tau)$ then $\Qop$ is strongly mixing with power rate on any space of $\alpha$-H\"older continuous functions $C^\alpha(\T^d)$. More precisely, there are $C<\infty$ and $p>0$ such that
	$$ 
	\bnorm{\Qop^n\varphi -\int \varphi dm}_{\infty} \leq C \norm{\varphi}_\alpha \frac{1}{n^p}  \quad \forall \, \varphi \in C^\alpha(\T^d), n \ge 1.
	$$
	In fact, $p$ can be chosen  $\frac{\alpha}{\tau} -\iota$, for any $\iota > 0$, in which case $C$ will depend on $\iota$.
\end{proposition}

\begin{proof}
Fix an observable $\varphi \in C^\alpha(\T^d)$. The trick for obtaining a sharp rate of convergence is to approximate $\varphi$ by trigonometric polynomials, with an error bound (and algebraic complexity) correlated to the number of iterates of the Markov operator. 

Fix $n$ to be this number of iterates. Let $N$ be the degree of approxi\-mation which will be chosen later. Since $\varphi \in C^\alpha(\T^d)$, by Jackson's approximation theorem (see~\cite[Section 2.2 and Section 3]{KLM} and references therein for more details) there exists a trigonometric polynomial $p_N$, with  $\deg p_N\leq N$, such that for some universal constant $C_0 < \infty$, 
$$\norm{\varphi-p_N}_{\infty} \le C_0 \,  \norm{\varphi}_\alpha \, \frac{1}{N^\alpha} \, . $$

In fact $p_N$ is the convolution of $\varphi$ with the Jackson kernel, so 
$$p_N=\sum_{\abs{k} \le N}c_k e_k$$ where the coefficients $c_k$ satisfy
	$$
	\abs{c_k}\leq\abs{\hat{\varphi}(k)}\le \norm{\varphi}_\alpha  \, .
	$$
	
We can then write $$\varphi=p_N +(\varphi-p_N)=: p_N+r_N $$ 
and by linearity we have
	$$
	\Qop^n\varphi=\Qop^np_N+\Qop^nr_N
	$$
	and
	$$
	\int \varphi dm=\int p_N dm +\int r_N dm \, .
	$$
	Thus
	$$
	\Qop^n\varphi -\int \varphi dm=\Qop^np_N-\int p_N dm+\Qop^nr_N-\int r_Ndm,
	$$
	which shows that
	$$
	\bnorm{\Qop^n\varphi -\int \varphi dm}_{\infty} \leq \bnorm{\Qop^np_N-\int p_N dm}_{\infty} +\bnorm{\Qop^nr_N}_{\infty}+\int\abs{r_N}dm.
	$$
	
	Due to the bound on $r_N = \varphi - p_N$ and the fact that $\Qop$ is a bounded operator with norm $1$ on $C^0 (\T^d)$, the second and third term on the right-hand side above are smaller than $C_0 \norm{\varphi}_\alpha \frac{1}{N^\alpha}$. It remains to estimate the first term.
	Since $p_N=\sum_{\abs{k}\leq N}c_ke_k$,
	$$
	\Qop^np_N=  \sum_{\abs{k}\leq N} c_k \, \Qop^n e_k =  \sum_{\abs{k}\leq N} c_k \,  \hat{\mu}(k)^n \, e_k.
	$$

This implies
\begin{align*}	
	\bnorm{Q^np_N-\int p_N dm}_{\infty} & \le \sum_{0<\abs{k}\leq N}\abs{c_k}\abs{\hat{\mu}(k)}^n\le \norm{\varphi}_\alpha\sum_{0<\abs{k}\leq N}(1-\frac{\gamma}{\abs{k}^\tau})^n \\
	& \le \norm{\varphi}_\alpha (2 N)^d \, (1-\frac{\gamma}{N^\tau})^n \, .
\end{align*}
	
Using the inequality $(1-x)^{\frac{1}{x}}\leq e^{-1}$ for  $x \in (0, 1)$, we have (for $N$ large enough)
	$$
 (1-\frac{\gamma}{N^\tau})^n \leq e^{-\frac{n\gamma}{N^\tau}} \, .
	$$

Combining the above estimates we obtain
$$\bnorm{\Qop^n\varphi -\int \varphi dm}_{\infty} \le 2^d \, \norm{\varphi}_\alpha \, N^d \, e^{-\frac{n\gamma}{N^\tau}} + 2 C_0 \,  \norm{\varphi}_\alpha \, \frac{1}{N^\alpha} \, .$$

Fix any $\ep > 0$ and choose $N :=(n\gamma)^{\frac{1-\ep}{\tau}}$. Then
$$\frac{1}{N^\alpha}  = \frac{1}{\gamma^{(\alpha/\tau) \, (1-\ep)}} \,  \frac{1}{n^{(\alpha/\tau) \, (1-\ep)}} =: C_1 \, \frac{1}{n^p} \, ,$$
where $p := \frac{\alpha}{\tau} \, (1-\ep) = \frac{\alpha}{\tau} - o (1)$,
while	
$$N^d \, e^{-\frac{n\gamma}{N^\tau}} = (n\gamma)^{\frac{(1-\ep) \, d}{\tau}} \, e^{- (n \gamma)^\ep} \ll \frac{1}{n^p} $$
for $n$ large enough.
This completes the proof provided the constant $C$ is chosen large enough depending on $\alpha, \gamma, \tau, d$ and $\ep$.
\end{proof}


\section{Mixed random-quasiperiodic dynamical systems}
\label{randtransl}

In this section we introduce a type of partially hyperbolic dynamical system for which the abstract statistical properties from Section~\ref{aldts} are (eventually) applicable and which moreover does not belong to the classes of non uniformly hyperbolic systems considered in~\cite{Chazottes-concentration, Melbourne-Nicol, Melbourne-PAMS, Alves-Freitas-L-V}.

Let $\Sigma:=\T^d$ where $\mu\in \Prob(\T^d)$ and regard $(\Sigma, \mu)$ as a probability space of symbols. Consider $X := \Sigma^\Z$, $\mu^\Z \in \Prob(X)$ the product space with the product measure and let $\sigma \colon X \to X$ be the bilateral Bernoulli shift. Define the skew-product map
$$
f \colon X\times \T^d\to X \times \T^d, \quad f(\omega,\theta) =(\sigma \omega, \theta+\omega_0) \, .
$$
We call the MPDS $(X\times \T^d, f, \mu^\Z\times m)$ a mixed random-quasiperiodic dynamical system (see~\cite{CDK-paper1} for more details). Next we define a general space of observables.

\begin{definition}\label{Holder space}
Given $\alpha \in (0, 1]$, $\Hscr_\alpha(X\times \T^d)$ is the set of observables $\varphi \colon X \times \T^d \to \R$ such that $v_\alpha(\varphi)=v_\alpha^X(\varphi)+v_\alpha^{\T^d}(\varphi)<\infty$, where
$$
v_\alpha^X(\varphi):=\sup_{\theta\in \T^d}\sup_{\omega\neq \omega'}\frac{\abs{\varphi(\omega,\theta)-\varphi(\omega',\theta)}}{d(\omega,\omega')^\alpha}
$$
and
$$
v_\alpha^{\T^d}(\varphi):=\sup_{\omega\in X}\sup_{\theta\neq \theta'}\frac{\abs{\varphi(\omega,\theta)-\varphi(\omega,\theta')}}{\abs{\theta-\theta'}^\alpha}.
$$
Note that $v_\alpha$ is a semi-norm. Endowed with the norm
$$
\norm{\varphi}_\alpha:=v_\alpha(\varphi)+\norm{\varphi}_\infty
$$
the set $\Hscr_\alpha(X\times \T^d)$ of observables is then a Banach space.
\end{definition}

\begin{remark}
As usual, the metric $d$ on $X$ is given by  
$
d(\omega,\omega'):=2^{-\min\{\abs{j} \colon \omega_j\neq\omega'_j\}}$ for  $\omega,\omega'\in X$.
Note that in general this metric does not make $(X,d)$ a compact metric space unless $\mu$ is finitely supported. This is essentially due to the fact  that this metric only accounts for where two points in $X$ differ, without telling by how much they differ, thus it does not (in general) metrize the product topology. However, the space of observables $\Hscr_\alpha(X\times \T^d)$ defined above relative to this metric contains the space of $\alpha$-H\"older continuous observables on $X\times \T^d$ with respect to the standard (compact) metric.
\end{remark}

The Markov chain on $X\times\T^d$
$$
(\omega, \theta)\to (\sigma\omega,\theta+\omega_0)\to(\sigma^2\omega,\theta+\omega_0+\omega_1)\to \cdots
$$
is evidently not strongly mixing because it is deterministic, hence we cannot derive LDT estimates and a CLT directly, via the two abstract theorems in Section~\ref{aldts}. 
We will proceed through three different levels of generality regarding the observables considered.

\subsection{First level: locally constant observables}
Let $\mu\in \Prob(\T^d)$. Let $K \colon \T^d \to \Prob(\T^d), K_{\theta}=\int \delta_{\theta+\omega_0} d\mu(\omega_0)$ be a Markov kernel and let
$\Qop \colon C^0(\T^d)\to C^0(\T^d), \Qop \varphi(\theta)=\int \varphi(\theta+\omega_0)d\mu(\omega_0)$ be the corresponding Markov operator. Finally, let $Z_0=\theta, Z_j=\theta+\omega_0+\cdots+\omega_{j-1}, j \ge 1$ be the corresponding $K$-Markov chain on $\T^d$.

We proved that if $\mu$ satisfies a mixing DC then $(\T^d, K, m, C^\alpha(\T^d))$ is a strongly mixing Markov system with decaying rate $r_n=\frac{1}{n^p}$, $n \ge 1$ and $p=\frac{\alpha}{\tau} -$. 
By the abstract LDT Theorem \ref{abstractldt} and Remark \ref{strongldt}, we obtain the following effective LDT estimate in this setting.

\begin{theorem}
	Assume that $\mu\in \rm{DC}(\gamma,\tau)$. Then $\forall\, \varphi \in C^\alpha(\T^d)$, $\forall\, \theta\in \T^d$ and $\forall\, \epsilon>0$, there is $n (\ep) \in \N$ such that $\forall n \ge n (\ep)$ we have
	$$
	\mu^\N\left\{\abs{\frac{1}{n}S_n\varphi - \int \varphi dm}>\epsilon\right\}< e^{-c(\epsilon)n}
	$$
where  $c(\epsilon) = \cldt \, \epsilon^{2+\frac{1}{p}}$, $n (\ep) = \nldt \, \ep^{- \frac{1}{p}}$ for constants $\cldt > 0$ and $\nldt \in \N$ which depend explicitly and uniformly on the data and $p = \frac{\alpha}{\tau} -$.
\end{theorem}

It is straightforward to check that for $(\T^d, K, m, C^\alpha(\T^d))$, the assumptions of Proposition~\ref{corclt} are satisfied via Lemma 2.2 of \cite{CDK-paper1} and thus we obtain the following.
\begin{theorem}
	Assume that $\mu\in \rm{DC}(\gamma,\tau)$ and let $\alpha>\tau$. Then for every $\varphi \in C^\alpha(\T^d)$ nonzero with zero mean, there exists $\sigma = \sigma(\varphi)>0$ such that
	$$
	\frac{S_n\varphi}{\sigma \, \sqrt n} \stackrel{d}{\longrightarrow} \mathcal{N}(0,1) \, .
	$$

\end{theorem}

	
\medskip

Next we extend to a slightly more general setup. On $\Sigma \times\T^d$ consider the Markov kernel $\bar K \colon \Sigma \times\T^d \to \Prob (\Sigma \times\T^d)$, 
$$
\bar{K}_{(\omega_0,\theta)}:=\int \delta_{(\omega_1, \theta+\omega_0)}d\mu(\omega_1)
$$
and the corresponding Markov operator $\bar{\Qop}: C^0(\Sigma\times \T^d)\to C^0(\Sigma\times \T^d)$
$$
\bar{\Qop}\varphi(\omega_0,\theta)=\int \varphi(\omega_1,\theta+\omega_0)d\mu(\omega_1).
$$

The corresponding $\bar{K}$-Markov chain is 
$$
Z_0=(\omega_0,\theta)\to Z_1=(\omega_1, \theta+\omega_0)\to Z_2=(\omega_2,\theta+\omega_0+\omega_1)\to \cdots
$$

Define $\Pi: C^0(\Sigma \times \T^d)\to C^0(\T^d)$, $\Pi\varphi(\theta)=\int \varphi(\omega_0,\theta)d\mu(\omega_0)$. It is clear that $\bar{\Qop}\varphi(\omega_0,\theta)=\Pi\varphi(\theta+\omega_0)$. By induction, $$\bar{\Qop}^n\varphi(\omega_0,\theta)=\Qop^{n-1}(\Pi\varphi)(\theta+\omega_0).$$

Define the space $\Hscr_{0,\alpha}(\Sigma \times \T^d))$ as follows:
$$
\Hscr_{0,\alpha}(\Sigma \times \T^d)):=\left\{\varphi \in C^0(\Sigma \times \T^d): v_\alpha^{\T^d}(\varphi)<\infty\right\}
$$
where
$$
v_\alpha^{\T^d}(\varphi):=\sup_{\omega_0\in \Sigma}\sup_{\theta\neq \theta'}\frac{\abs{\varphi(\omega_0,\theta)-\varphi(\omega_0,\theta')}}{\abs{\theta-\theta'}^\alpha}.
$$
The corresponding $\alpha$-norm is defined by $\norm{\varphi}_\alpha=\norm{\varphi}_\infty+v_\alpha^{\T^d}(\varphi)$.
Then it is straightforward to check that $(\Sigma\times \T^d, \bar{K}, \mu \times m, \Hscr_{0,\alpha}(\Sigma \times \T^d))$ is a Markov system. Since $\Qop$ is strongly mixing on $C^\alpha(\T^d)$, then $\bar{\Qop}$ is strongly mixing on $\Hscr_{0,\alpha}(\Sigma \times \T^d)$ with the same decaying rate $r_n=\frac{1}{n^p}$ because $\Pi \Hscr_{0,\alpha}(\Sigma \times \T^d)\subset C^\alpha(\T^d)$. We then get the following LDT estimate and, if $\alpha>\tau$, CLT. 

\begin{theorem}
	Assume that $\mu\in \rm{DC}(\gamma,\tau)$. Then $\forall\, \varphi \in \Hscr_{0,\alpha}(\Sigma\times\T^d)$, $\forall\, (\omega_0,\theta)\in \Sigma\times \T^d$ and $\forall\, \epsilon>0$ there is $n (\ep) \in \N$ such that $\forall n \ge n (\ep)$ we have
	$$
	\Pp_{(\omega_0,\theta)}\left\{\abs{\frac{1}{n}S_n\varphi-\int \varphi d\mu\times m}>\epsilon\right\}<e^{-c(\epsilon)n}
	$$
where $c(\epsilon) = \cldt \, \epsilon^{2+\frac{1}{p}}$, $n (\ep) = \nldt \, \ep^{- \frac{1}{p}}$ and $p = \frac{\alpha}{\tau}-$.
\end{theorem}

Again, for $(\Sigma\times \T^d, \bar{K}, \mu \times m, \Hscr_{0,\alpha}(\Sigma \times \T^d))$ the conditions of Proposition~\ref{corclt} are fulfilled by using Lemma 2.2 in \cite{CDK-paper1}, and we obtain the following CLT.

\begin{theorem}
	Assume that $\mu\in \rm{DC}(\gamma,\tau)$ with $\alpha>\tau$. Then $\forall\, \varphi \in \Hscr_{0,\alpha}(\Sigma\times\T^d)$ with zero mean and nonzero $L^1 (\mu \times m)$-norm, there exists $\sigma = \sigma(\varphi)>0$ such that
	$$
	\frac{S_n\varphi}{\sigma \, \sqrt n} \stackrel{d}{\longrightarrow} \mathcal{N}(0,1) \, .
	$$
\end{theorem}

\medskip
 
\subsection{Second level: future independent observables}
Let $X^-:=\Sigma^{-\N}=\left\{ \omega^-=\{\omega_j\}_{j\leq 0}: \omega_j \in \Sigma\right\}$ endowed with the distance $d$ defined above and  denote by $\mu^{-\N}$ the product measure on $X^-$.
The Markov kernel $K^-$ on $X^-\times \T^d$ is defined by
$$
K^-_{(\omega^-,\theta)}=\int \delta_{(\omega^-\omega_1, \theta+\omega_0)}d\mu(\omega_1)
$$ 
and the corresponding Markov operator $\Qop^-$ on $C^0(X^-\times \T^d)$ is 
$$
\Qop^-\varphi(\omega^-,\theta)=\int \varphi(\omega^-\omega_1,\theta+\omega_0)d\mu(\omega_1).
$$
The associated Markov chain is 
$$
(\omega^-,\theta)\to (\omega^-\omega_1,\theta+\omega_0)\to(\omega^-\omega_1\omega_2,\theta+\omega_0+\omega_1)\to \cdots
$$

The corresponding space of $\alpha$-H\"older observables, denoted by $\Hscr_\alpha(X^-\times \T^d)$, is defined as
$$
\Hscr_\alpha(X^-\times \T^d):=\left\{\varphi\in C^0(X^-\times \T^d) \colon v_\alpha(\varphi):= v_\alpha^{X^-}(\varphi)+v_\alpha^{\T^d}(\varphi)<\infty\right\}
$$
where
$$
v_\alpha^{X^-}(\varphi):=\sup_{\theta\in \T^d}\sup_{\omega^-\neq \omega'^-}\frac{\abs{\varphi(\omega^-,\theta)-\varphi(\omega'^-,\theta)}}{d(\omega^-,\omega'^-)^\alpha} 
$$
and
$$
v_\alpha^{\T^d}(\varphi):=\sup_{\omega^-\in X^-}\sup_{\theta\neq \theta'}\frac{\abs{\varphi(\omega^-,\theta)-\varphi(\omega^-,\theta')}}{\abs{\theta-\theta'}^\alpha} \, .
$$
Endowed with the norm  $\norm{\varphi}_\alpha:=\norm{\varphi}_\infty+v_\alpha(\varphi)$, $\Hscr_\alpha(X^-\times \T^d)$ is a Banach space.

It is not difficult to verify that $(X^-\times \T^d, K^-, \mu^{-\N} \times m, \Hscr_\alpha(X^-\times \T^d))$ is a Markov system.

\subsection*{Contracting factors}
Our goal is now to show that the observed SDS $(X^-\times \T^d, K^-, \mu^{-\N} \times m, \Hscr_\alpha(X^-\times \T^d))$ is strongly mixing with rate $r_n = \frac{1}{n^p},\, p=\frac{\alpha}{\tau}-$. We have already shown that $(\Sigma\times \T^d,\bar{K}, \mu\times m, \Hscr_{0, \alpha} (\Sigma\times \T^d))$ is strongly mixing with rate $r_n = \frac{1}{n^p}, \,p=\frac{\alpha}{\tau}-$. To this end we will prove that $(\Sigma\times \T^d,\bar{K}, \mu\times m, \Hscr_{0, \alpha} (\Sigma\times \T^d))$ is a contracting factor (see~Definition~\ref{confac}) of $(X^-\times \T^d, K^-, \mu^{-\N} \times m, \Hscr_\alpha(X^-\times \T^d))$, which will allow us to lift the strong mixing property from $\bar{K}$ to $K^-$ (see Theorem~\ref{liftsm}). 

We first introduce the definition of a factor.

\begin{definition}\label{factor}
	Given two Markov systems $(M, K,\mu,\Escr)$, $(\tilde{M},\tilde{K},\tilde{\mu},\tilde{\Escr})$, the first is called a factor of the second if there exists a continuous projection $\pi \colon \tilde{M} \to M$ such that denoting by $\pi_\ast \colon \Prob (\tilde{M}) \to \Prob (M)$ the push-forward operator $\pi_\ast \tilde{\nu} := \tilde{\nu} \circ \pi^{-1}$ and by $\pi^\ast \colon C^0 (M) \to C^0 (\tilde{M})$ the pull-back operator $\pi^\ast (\varphi) := \varphi \circ \pi$, the following hold:
	\begin{enumerate}
		\item $\pi_\ast \tilde{\mu}=\mu$,
		\item $K_{\pi(\tilde{x})}=\pi_\ast \tilde{K}_{\tilde{x}}$ for all $\tilde{x}\in \tilde{M}$,
		\item  there exists $\eta \colon M\to\tilde{M}$ continuous with
		$\pi\circ\eta=\id_M$ such that   $\eta^\ast (\tilde{\Escr})\subseteq \Escr$
		and $\norm{\varphi\circ \eta}_\Escr\leq M_1\norm{\varphi}_{\tilde{\Escr}}$
		for some constant $M_1<\infty$ and all  $\varphi\in \tilde{\Escr}$,
		\item  $\pi^\ast(\Escr)\subseteq \tilde{\Escr}$ and   $\norm{ \varphi\circ \pi }_{\tilde{\Escr}}\leq M_2\, \norm{\varphi}_\Escr$ for some constant $M_2<\infty$ and all   $\varphi\in \Escr$.		
	\end{enumerate}
\end{definition}

Factors have the following properties.

\begin{proposition}
	Let $(M,K,\mu,\Escr)$ be a factor of $(\tilde{M},\tilde{K},\tilde{\mu},\tilde{\Escr})$. Then
	\begin{enumerate}
		\item $\pi^\ast \circ \Qop_K=\Qop_{\tilde{K}} \circ \pi^\ast$, i.e. the following commutative diagram holds
		$$  \begin{CD}
			C^0(\tilde{M})    @>\Qop_{\tilde{K}}>>  C^0(\tilde{M}) \\
			@A\pi^\ast AA        @AA\pi^\ast A\\
			C^0(M)     @>>\Qop_{K}>  C^0(M)
		\end{CD}  $$
		\item The bounded linear map $\pi^\ast \colon \Escr \to \pi^\ast(\Escr)$ is an isomorphism onto the closed linear subspace $\pi^\ast(\Escr)\subseteq \tilde{\Escr}$.
	\end{enumerate}
\end{proposition}
\begin{proof}
	Since $K_{\pi(\tilde{x})}= \pi_\ast \tilde{K}_{\tilde{x}}$,
	$$
	(\Qop_{\tilde{K}} \circ \pi^\ast \varphi)(\tilde{x})=\int \varphi \circ \pi d\tilde{K}_{\tilde{x}} =\int \varphi dK_{\pi(\tilde{x})}=(\pi^\ast \circ \Qop_K \varphi)(\tilde{x})
	$$
which proves item (1).

Let us prove item (2). By definition, $\pi^\ast \varphi=\varphi \circ \pi$ is a bounded linear operator. Since $\pi$ is surjective, $\pi^\ast$ is one to one. Thus $\pi^\ast :\Escr \to \pi^\ast(\Escr)$ is a linear bijection.
	Define the closed linear subspace of $\tilde{\Escr}$
	$$
	\mathcal{V}:=\{\varphi \in \tilde{\Escr}: \forall \, x, y \in \tilde{M}, \pi(x)=\pi(y)\Rightarrow \varphi(x)=\varphi(y)\}.
	$$
	The linearity of $\mathcal{V}$ is clear. For the closedness, assume that $\tilde{\varphi}_n\in \mathcal{V}$ and $\tilde{\varphi}_n \to \tilde{\varphi}$ pointwise in $\tilde{\Escr}$. If $\pi(x)=\pi(y)$, then $\tilde{\varphi}_n(\tilde{x})=\tilde{\varphi}_n(\tilde{y})$. Letting $n\to \infty$ we get $\tilde{\varphi}(\tilde{x})=\tilde{\varphi}(\tilde{y})$ with $\tilde{\varphi}\in \tilde{\Escr}$, which shows that $\tilde{\varphi}\in \mathcal{V}$.
	
	Clearly $\pi^\ast(\Escr)\subseteq \mathcal{V}$. Conversely, given $\varphi\in \mathcal{V}$ consider the function $\psi:= \varphi \circ \eta \in \Escr$ where $\eta: M \to \tilde{M}$ is given by Definition~\ref{factor}. Since $\pi(x)=\pi(\eta(\pi(x)))$, by the definition of $\mathcal{V}$ we have $\varphi(x)=\varphi(\eta(\pi(x)))$ for all $x\in \tilde{M}$, which proves that $\varphi=\psi\circ \pi \in \pi^\ast(\Escr)$. Therefore, $\mathcal{V}=\pi^\ast(\Escr)$ is a closed linear subspace of $\tilde{\Escr}$, thus also a Banach (sub)space with the induced norm from $\tilde{\Escr}$. Finally, by the Banach open mapping theorem, $\pi^\ast$ is an open map. Thus the inverse map $(\pi^\ast)^{-1}: \pi^\ast (\Escr) \to \Escr$ is continuous, so it is also a bounded linear map. This proves that $\pi^\ast$ is an isomorphism.
\end{proof}

We introduce the notion of contracting factors.
\begin{definition}\label{confac}
	We call $(M,K,\mu,\Escr)$ a contracting factor of $(\tilde{M},\tilde{K},\tilde{\mu},\tilde{\Escr})$ with contracting rate $\tau = \{\tau_n\}_{n\ge1}$ if additionally we have the following: $\exists\, C < \infty$ such that $\forall\, \tilde{\varphi}\in \tilde{\Escr}$, $\forall n\in \N$, $\exists\, \psi_n\in \Escr,$ satisfying,   for all $n\in \N$,
	$$
	\norm{\psi_n}_\infty \leq \norm{\tilde{\varphi}}_\infty, \,\,\, \norm{\psi_n}_\Escr \leq C \norm{\tilde{\varphi}}_{\tilde{\Escr}}
	$$
	and
	$$
	\norm{\tilde{\Qop}^n\tilde{\varphi}-\pi^\ast \psi_n}_\infty \leq C\norm{\tilde{\varphi}}_{\tilde{\Escr}} \, \tau_n \, .
	$$
	\end{definition}

We have the following abstract result.

\begin{theorem}\label{liftsm}
	Assume that $(M,K,\mu,\Escr)$ is strongly mixing with rate $r$ and that $(M,K,\mu,\Escr)$ is a contracting factor of $(\tilde{M},\tilde{K},\tilde{\mu},\tilde{\Escr})$ with contracting rate $\tau$. Then $(\tilde{M},\tilde{K},\tilde{\mu},\tilde{\Escr})$ is strongly mixing with rate $r^\ast_n=\max\{r_{\frac{n}{2}},\tau_{\frac{n}{2}}\}$.
\end{theorem}
\begin{proof}
	Fix $\tilde{\varphi}\in \tilde{\Escr}$ and $n\in \N$. We may assume that $n$ is even. Otherwise, since $\tilde{\Qop}^n\tilde{\varphi}=\tilde{\Qop}^{n-1}(\tilde{\Qop}\tilde{\varphi})$, we can work with $\tilde{\Qop}\tilde{\varphi}$ instead of $\tilde{\varphi}$. For this $\tilde{\varphi}$ and $\frac{n}{2}$, consider $\psi_{\frac{n}{2}}=:\psi\in \Escr$ such that 
	$$
	\norm{\psi}_\infty \leq \norm{\tilde{\varphi}}_\infty, \,\, \norm{\psi}_\Escr \lesssim \norm{\tilde{\varphi}}_{\tilde{\Escr}}
	$$
	and
	$$
	\norm{\tilde{\Qop}^{\frac{n}{2}}\tilde{\varphi}-\pi^\ast\psi}_\infty \lesssim \norm{\tilde{\varphi}}_{\tilde{\Escr}} \, \tau_{\frac{n}{2}} ,
	$$
where we write $a \lesssim b$ if there is a context universal constant $C < \infty$ such that $a \le C \, b$.
	
Since $\tilde{\mu}$ is $\tilde{K}$-stationary we have
	$$
	\int \tilde{\Qop}^j\tilde{\varphi}d\tilde{\mu}=\int \tilde{\varphi} d\tilde{\mu}, \ \forall\, j\in \N ,
	$$

and since $\pi_\ast \tilde{\mu}=\mu$, we have
	$$
	\int \pi^\ast \psi d\tilde{\mu}=\int \psi d\mu.
	$$
	
Thus integrating both sides of the last inequality we obtain
	$$
	\abs{\int \tilde{\varphi}d\tilde{\mu}-\int \psi d\mu}\lesssim  \norm{\tilde{\varphi}}_{\tilde{\Escr}} \, \tau_{\frac{n}{2}} \, .
	$$
	
Using that $(M,K,\mu,\Escr)$ is strongly mixing with rate $r$, we have
	$$
	\bnorm{\Qop^{\frac{n}{2}}\psi-\int \psi d\mu}_\infty \lesssim \norm{\psi}_\Escr \, r_{\frac{n}{2}}
	\lesssim \norm{\tilde{\varphi}}_{\tilde{\Escr}} \, r_{\frac{n}{2}} \, .
	$$
	
On the other hand, by the commutativity of the diagram,
	$$
	\tilde{\Qop}^{\frac{n}{2}}\pi^\ast\psi=\pi^\ast \Qop^{\frac{n}{2}}\psi.
	$$
	
	Treating $\int \psi d\mu$ as a constant function, we have $\pi^\ast(\int \psi d\mu)=\int \psi d\mu$, so
	$$
	\bnorm{\tilde{\Qop}^{\frac{n}{2}}\pi^\ast\psi-\int \psi d\mu}_\infty= \bnorm{\pi^\ast\tilde{\Qop}^{\frac{n}{2}}\psi-\pi^\ast(\int \psi d\mu)}_\infty \lesssim \bnorm{\tilde{\varphi}}_{\tilde{\Escr}} \,  r_{\frac{n}{2}} \, . $$
	
Finally, note that
	$$
	\tilde{\Qop}^n \tilde{\varphi}-\int \tilde{\varphi}d\tilde{\mu}=\tilde{\Qop}^n\tilde{\varphi}-\tilde{\Qop}^{\frac{n}{2}}(\pi^\ast\psi)+\tilde{\Qop}^{\frac{n}{2}}(\pi^\ast\psi)-\int \psi d\mu+\int \psi d\mu-\int \tilde{\varphi} d\tilde{\mu}.
	$$
	
	Thus by the triangle inequality,
	$$
	\bnorm{\tilde{\Qop}^n \tilde{\varphi}-\int \tilde{\varphi}d\tilde{\mu}}_\infty \lesssim \norm{\tilde{\varphi}}_{\tilde{\Escr}} 
	\left(\tau_{\frac{n}{2}} + r_{\frac{n}{2}} + \tau_{\frac{n}{2}} \right)
	$$
and the result follows.
\end{proof}

We apply this abstract result with $(M,K,\mu,\Escr)$ being the Markov system $(\Sigma \times \T^d, \bar{K}, \mu \times m, \Hscr_{0,\alpha}(\Sigma \times \T^d))$, which was shown to be strongly mixing with rate $r_n=\frac{1}{n^p}$, where $p = \frac{\alpha}{\tau}-$, provided that the measure $\mu$ is mixing DC$(\gamma,\tau)$. Moreover,  $(\tilde{M},\tilde{K},\tilde{\mu},\tilde{\Escr})$ is the Markov system $(X^-\times \T^d, K^-,\mu^{-\N}\times m, \Hscr_\alpha(X^-\times \T^d))$.

\begin{theorem}
	$(\Sigma \times \T^d, \bar{K}, \mu \times m, \Hscr_{0,\alpha}(\Sigma \times \T^d))$ is a contracting factor  of $(X^-\times \T^d, K^-,\mu^{-\N}\times m, \Hscr_\alpha(X^-\times \T^d))$ with exponential rate. Therefore, the latter is strongly mixing with rate $r_n = \frac{1}{n^p}$.
\end{theorem}

\begin{proof}
	Define $\pi \colon X^-\times \T^d \to \Sigma \times \T^d, \,\pi(\omega^-,\theta)=(\omega_0,\theta)$. Fix $a\in \Sigma$, define $\eta \colon \Sigma \times \T^d \to X^-\times \T^d, \, \eta(\omega_0,\theta)=(\cdots aa\omega_0,\theta)$. It is straightforward to check items (1)-(4) in Definition \ref{factor}, thus the first system is a factor of the second. It remains to show that it is a contracting factor.
	
	Fix $n\in \N$ and denote by $\Hscr_{\alpha,n}(X^-\times \T^d)$ the observables in the space $\Hscr_\alpha(X^-\times \T^d)$ that only depend on the last $n$ random coordinates $\omega_{-n+1},\cdots, \omega_{-1},\omega_0$ and on $\theta$. 
If $\varphi \in \Hscr_{\alpha,n}(X^-\times \T^d)$, then
\begin{align*}
(\Qop^-)^n \varphi(\omega^-,\theta) & \\
&\kern-4em =\int\cdots \int \varphi(\omega^-\omega_1\ldots\omega_n,\theta+\omega_0+\cdots+\omega_{n-1})d\mu(\omega_n)\cdots d\mu(\omega_1)
\end{align*}
	only depends on $(\omega_0,\theta)$, so $(\Qop^-)^n \varphi(\omega^-,\theta)\in \Hscr_{0,\alpha}(\Sigma \times \T^d)$.
	
	Fix any $\varphi \in \Hscr_\alpha(X^-\times \T^d)$ and $n\in \N$. We construct $\psi_n\in \Hscr_{0,\alpha}(\Sigma \times \T^d)$ as follows. Let $ \varphi_n := \varphi \circ i_n$, where $i_n(\omega^-,\theta) :=(\cdots aa\omega_{-n+1}\cdots \omega_0,\theta)$.
Note that $\varphi_n\in \Hscr_{\alpha,n}(X^-\times \T^d)$. We proved that the bounded linear map $\pi^\ast \colon \Escr \to \pi^\ast(\Escr)$ is an isomorphism onto the closed linear subspace $\pi^\ast(\Escr)\subseteq \tilde{\Escr}$. So let $\psi_n\in \Hscr_{0,\alpha}(\Sigma\times\T^d)$ be such that 
	$$
	\pi^\ast \psi_n=\psi_n \circ \pi= (\Qop^-)^n\varphi_n.
	$$
	Then
	\begin{align*}
		\norm{(\Qop^-)^n(\varphi)-\psi_n\circ \pi}_\infty & = \norm{(\Qop^-)^n(\varphi)-(\Qop^-)^n\varphi_n}_\infty
	 \le  \norm{\varphi-\varphi_n}_\infty \\
	 & = \sup_{(\omega^-,\theta)\in X^-\times \T^d}\abs{\varphi(\omega^-,\theta)-\varphi(i_n(\omega^-,\theta))}\\
        & \leq  v_\alpha^{X^-}(\varphi) \, d\left(\omega^-, (\cdots aa\omega_{-n+1}\cdots\omega_0) \right)^\alpha\\
	&	\leq 2^{-n\alpha}\norm{\varphi}_\alpha.
	\end{align*}
Let $\sigma=2^{-\alpha}<1$ and conclude the proof by applying Theorem \ref{liftsm}.
\end{proof}

Fix any $\varphi \in \Hscr_\alpha(X^-\times \T^d)$ and $(\omega^-,\theta)\in X^-\times \T^d$ and consider the $K^-$-Markov chain $\{Z_n\}_{n\geq 0}$ such that 
$$
Z_0=(\omega^-,\theta), \ 
Z_n=(\omega^-\omega_1\cdots\omega_n,\theta+\omega_0+\cdots+\omega_{n-1}) \, \forall n\ge 1 .
$$
Let $S_n\varphi=\varphi(Z_0)+\cdots+\varphi(Z_{n-1})$.
From the abstract LDT Theorem~\ref{abstractldt} and Corollary~\ref{strongldt}, we obtain the following.

\begin{theorem}
	Assume that $\mu\in \rm{DC}(\gamma,\tau)$. Then $\forall\, \varphi \in \Hscr_\alpha(X^-\times \T^d)$, $\forall\, (\omega^-,\theta)\in X^-\times \T^d$ and $\forall\,\epsilon>0$ there is $n (\ep) \in \N$ such that $\forall n\ge n (\ep)$ we have
	$$
	\Pp_{(\omega^-,\theta)}\left\{\abs{\frac{1}{n}S_n\varphi-\int_{X^-\times \T^d} \varphi \, d\mu^{-\N}\times m}>\epsilon\right\}< e^{-c(\epsilon)n}
	$$
where $c(\epsilon) = \cldt \, \epsilon^{2+\frac{1}{p}}$, $n (\ep) = \nldt \, \ep^{- \frac{1}{p}}$ and $p = \frac{\alpha}{\tau}-$.
\end{theorem}

Let us slightly reformulate the LDT estimate above in more dynamical terms. Consider the probability space $(X\times \T^d, \mu^\Z \times m)$ and the transformation $f \colon X\times \T^d \to X\times \T^d, f(\omega,\theta)=(\sigma\omega,\theta+\omega_0)$. Define the projection $\pi \colon X\times \T^d \to X^-\times \T^d, \pi(\omega,\theta)=(\omega^-,\theta)$ and note that $\pi_\ast(\mu^\Z\times m)=\mu^{-\N}\times m$. Hence if $\bar{Z_0} = (\om, \theta) \in X \times \T^d$ is chosen according to the probability $\mu^\Z\times m$, then  $Z_0= \pi (\bar{Z_0}) = (\omega^-,\theta) \in X^-\times \T^d$ is chosen according to the probability $\mu^{-\N}\times m$. Moreover, the random process $Z_j =\pi ( f^j (\bar{Z}_0) )$ for $j \ge 0$ is a $K^-$-Markov chain, so by the previous theorem, $\forall\, \varphi \in \Hscr_\alpha(X^- \times \T^d)$ we have:
\begin{equation}\label{ldt10}
\mu^{\Z}\times m \Big\{(\omega,\theta) \colon \Big| \frac{1}{n}\sum_{j=0}^{n-1}\varphi \circ \pi (f^j(\omega, \theta))-\int \varphi d\mu^{-\N}\times m\Big|>\epsilon\Big\}< e^{-c(\epsilon)n} .
\end{equation}

If $\alpha>\tau$ we can choose $p \in (1, \frac{\alpha}{\tau})$. It is again easy to see that the conditions of Proposition~\ref{corclt} are satisfied for $(X^-\times \T^d, K^-,\mu^{-\N}\times m, \Hscr_\alpha(X^-\times \T^d))$ (using Lemma 2.2 of \cite{CDK-paper1}), thus implying the following.

\begin{theorem}\label{sigmaphi}
	Assume that $\mu\in \rm{DC}(\gamma,\tau)$ and let $\alpha>\tau$. Then $\forall\, \varphi \in \Hscr_\alpha(X^-\times \T^d)$ with zero mean and nonzero $L^1 (\mu^{-\N}\times m)$-norm, there exists $\sigma = \sigma(\varphi)>0$ such that
	$$
	\frac{S_n\varphi}{\sigma \, \sqrt n} \stackrel{d}{\longrightarrow} \mathcal{N}(0,1) \, .
	$$
\end{theorem}

\smallskip

\subsection{Third level: past and future dependent observables}

We extend the above statistical properties to observables in the space $\Hscr_\alpha(X\times\T^d)$  that also depend on the future. 	The idea  is to ``reduce'' an observable $\varphi\in \Hscr_\alpha(X\times \T^d)$ to an observable $\varphi^-\in \Hscr_\beta(X^-\times \T^d)$ for some $\beta >0$ (the precise statement is given in the proposition below).

We  identify the space  $\Hscr_\alpha(X^-\times \T^d)$ with the subspace of observables in $\Hscr_\alpha(X\times\T^d)$ that are future independent, where $\varphi$ is called future independent if $\varphi(x,\theta)=\varphi(y,\theta)$ whenever $x^-=y^-$. More precisely, an observable $\varphi \in \Hscr_\alpha(X^-\times \T^d)$ is identified with the future independent observable $\varphi \circ \pi \in \Hscr_\alpha(X\times\T^d)$, where $\pi \colon X\times \T^d \to X^-\times \T^d$ is the projection $\pi(\omega,\theta)=(\omega^-,\theta)$.

\begin{proposition}\label{rtpast}
Given $\varphi \in \Hscr_\alpha(X\times \T^d)$, there are $\varphi^-\in \Hscr_\beta(X^-\times \T^d)$ and $\eta\in \Hscr_\beta(X\times \T^d)$ with $\beta=\frac{\alpha}{2}$ such that
	\begin{equation}\label{reduction}
		\varphi-\varphi^-\circ f=\eta-\eta\circ f. 
	\end{equation}
	Moreover, the map $\Hscr_\alpha(X\times \T^d) \ni \varphi \mapsto \varphi^- \in \Hscr_\beta(X^-\times \T^d)$ is a bounded linear operator, that is, $\norm{\varphi^-}_\beta\lesssim \norm{\varphi}_\alpha$. 
\end{proposition}

We postpone for now the proof of this proposition in order to explain how it can be used to derive the LDT and CLT for the mixed dynamical system $(X\times\T^d, f)$.

Integrating both sides of the equation $(\ref{reduction})$ w.r.t. $\mu^\Z \times m$ we have
$$
\int \varphi \, d\mu^\Z \times m-\int \varphi^-\circ f \, d\mu^\Z \times m=\int \eta \, d\mu^\Z \times m-\int \eta\circ f \, d\mu^\Z \times m.
$$

Since $\mu^\Z \times m$ is $f$-invariant, the right hand side equals zero, so
$$
\int \varphi \, d\mu^\Z \times m=\int \varphi^- \circ f \, d\mu^\Z \times m=\int \varphi^- \, d\mu^\Z \times m=\int \varphi^- \, d\mu^{-\N} \times m.
$$

We write~\eqref{reduction} as
$$
\varphi=\varphi^-\circ f +\eta-\eta\circ f
$$
and consider the  Birkhoff sums of both sides:
$$
S_n\varphi=S_n(\varphi^- \circ f)+\eta -\eta\circ f^n \, .
$$

Then
$$
\frac{1}{n}S_n\varphi -\int \varphi d\mu^\Z \times m=\frac{1}{n}S_n(\varphi^-\circ f)-\int \varphi^- d\mu^{-\N} \times m+\frac{\eta-\eta \circ f^n}{n} \, .
$$
The last term on the right-hand side converges uniformly to zero, while the Birkhoff sums of $\varphi^-\circ f$ are essentially the same as those of $\varphi^-$.
Thus the LDT estimates on the observable $\varphi^- \in \Hscr_\beta(X^-\times \T^d))$ given by~\eqref{ldt10} transfer over  to the original observable $\varphi \in \Hscr_\alpha(X\times \T^d)$, establishing the following.

\begin{theorem}
	Assume that $\mu\in \rm{DC}(\gamma,\tau)$. Then for every observable $\varphi \in \Hscr_\alpha(X\times \T^d)$ and $\epsilon>0$ there is $n (\ep) \in \N$ such that for all $n\ge n (\ep)$ we have
	$$
	\mu^{\Z}\times m\left\{ \abs{\frac{1}{n}S_n\varphi -\int \varphi d\mu^\Z \times m}>\epsilon\right\}<e^{-c(\epsilon)n}
	$$
where $c(\epsilon) = \cldt \, \epsilon^{2+\frac{1}{p}}$, $n (\ep) = \nldt \, \ep^{- \frac{1}{p}}$ and $p= \frac{\alpha}{2\tau}-$.
\end{theorem}

In order to obtain a CLT we have to assume that $\alpha>2\tau$ so that $p = \frac{\beta}{\tau} - o(1)= \frac{\alpha}{2 \tau} - o(1) >1$. Moreover, if $\varphi$ has mean zero, then $\varphi^-$ also has mean zero. Furthermore, if $\varphi$ is not a coboundary (see Definition~\ref{coboundary def}), the relation~\eqref{reduction} clearly shows that $\varphi^-$ is not $\mu^{-\N} \times m$-a.e. zero, so Theorem~\ref{sigmaphi} is applicable.

Using the cohomological relation~\eqref{reduction} we have
$$
S_n\varphi =S_n(\varphi^- \circ f)+\eta -\eta\circ f^n ,
$$
so dividing both sides by $\sigma \sqrt{n}$ with $\sigma=\sigma(\varphi^-)>0$ (which further depends on $\varphi$) obtained from Theorem~\ref{sigmaphi}, we get
$$
\frac{S_n\varphi }{\sigma \sqrt{n}}= \frac{S_n(\varphi^- \circ f)}{\sigma \sqrt{n}}+\frac{\eta -\eta\circ f^n}{\sigma \sqrt{n}} .
$$
Because
$$\frac{\eta -\eta\circ f^n}{\sigma \sqrt{n}} \to 0 \quad \text{uniformly},$$
we conclude the equivalence of the CLT between 
$(X\times \T^d, f,\mu^\Z\times m, \Hscr_\alpha(X\times \T^d))$ and $(X^-\times \T^d, K^-,\mu^{-\N}\times m, \Hscr_\beta(X^-\times \T^d))$. Namely, 
$$
\frac{S_n(\varphi^- \circ f)}{\sigma \sqrt{n}} \stackrel{d}{\longrightarrow} \mathcal{N}(0,1), 
$$
if and only if
$$
\frac{S_n\varphi}{\sigma \sqrt{n}} \stackrel{d}{\longrightarrow} \mathcal{N}(0,1).
$$

\begin{theorem}
	Assume that $\mu\in \rm{DC}(\gamma,\tau)$ and let $\alpha>2\tau$. If $\varphi \in \Hscr_\alpha(X\times \T^d)$ has zero mean and it is not a coboundary then there exists $\sigma = \sigma (\varphi) >0$ such that
	$$
	\frac{S_n\varphi}{\sigma \, \sqrt n} \stackrel{d}{\longrightarrow} \mathcal{N}(0,1) \, .
	$$
\end{theorem}

Therefore it remains to prove Proposition \ref{rtpast}. Before that, let us make some preparations regarding the concepts of continuous disintegration and unstable holonomy.

\begin{definition}
	Let $(\tilde{M},\tilde{\mu})$ and $ (M,\mu)$ be measurable spaces. Assume that $\tilde{M}$ and $M$ are compact metric spaces and $\pi\colon \tilde{M} \to M$ is continuous with $\pi_\ast \tilde{\mu}=\mu$. A {\em continuous disintegration} of $\tilde{\mu}$ over $\pi$ is a family of probability measures $\{\tilde{\mu}_a\}_{a\in M}$ such that 
	\begin{enumerate}
		\item $\tilde{\mu}_a \in\Prob(\tilde{M})$ and $\tilde{\mu}_a(\pi^{-1}\{a\})=1$,
		\item $M\ni a \mapsto \tilde{\mu}_a \in \Prob(\tilde{M})$ is continuous,
		\item $\forall\, \varphi \in C^0(\tilde{M})$,
		$$
		\int_{\tilde{M}}\varphi d\tilde{\mu}=\int_M \left( \int_{\pi^{-1}\{a\}} \varphi d\tilde{\mu}_a \right) \, d\mu(a). 
		$$
	\end{enumerate}
\end{definition}

Let $\pi \colon X\times \T^d \to X^- \times \T^d, \pi(\omega,\theta):=(\omega^-,\theta)$ be the standard projection. For $\omega\in X$, we will write $\omega=(\omega^-;\omega^+)$ where $\omega^-\in X^-$ and $\omega^+\in X^+:=\Sigma^{\N^+}$. We then have $\pi_\ast(\mu^\Z \times m)=\mu^{-\N}\times m$.

For any $(\omega^-,\theta)\in X^-\times \T^d$, let
$$
\Pp_{(\omega^-,\theta)}:=\delta_{\omega^-} \times \mu^{\N^+}\times \delta_\theta \in \Prob(X\times \T^d).
$$
Then clearly we have that $\{\Pp_{(\omega^-,\theta)}\}_{(\omega^-,\theta)\in X^-\times \T^d}$ is a continuous disintegration of $\Pp=\mu^\Z \times m$ over $\pi$. Moreover, for $(\omega^-,\theta)\in X^-\times \T^d$,
\begin{align*}
\pi^{-1}\left\{(\omega^-,\theta)\right\} &=\left\{(\omega^-,\omega^+,\theta):\omega^+\in X^+\right\}\\
&= \{ (x,\theta)\in X\times \T^d \colon x^-= \omega^- \} =: W^u_{loc}(\omega,\theta)
\end{align*}
are the local unstable sets of the partially hyperbolic dynamical system $f \colon X \times \T^d\to X\times \T^d$. We clarify this in the following.

Let $x,y \in X$ with $x^-=y^-$. Namely,
$$
x=(\cdots,x_{-1},x_0;x_1,\cdots),\quad y=(\cdots, x_{-1},x_0;y_1,\cdots).
$$
Then
$$
\sigma^{-1}x=(\cdots, x_{-1}; x_0, x_1,\cdots), \quad \sigma^{-1}y=(\cdots, x_{-1};x_0,y_1,\cdots)
$$
which gives $d(\sigma^{-1}x,\sigma^{-1}y)\leq 2^{-1}$. If $(x,\theta), (y,\theta)$ belong to the same fiber $W^u_{loc}(x^-,\theta)$, then $x^-=y^-$ and
$$
f^{-1}(x,\theta)=(\sigma^{-1}x,\theta-x_{-1}), \quad f^{-1}(y,\theta)=(\sigma^{-1}y,\theta-x_{-1})
$$
are still in the same fiber $W^u_{loc}(x^-,\theta-x_{-1})$ with
$$
d (f^{-1}(x,\theta), f^{-1}(y,\theta))\leq 2^{-1}.
$$
So $f^{-1}$ contracts the fibers. By induction,
$$
d (f^{-n}(x,\theta), f^{-n}(y,\theta))\leq 2^{-n}.
$$
The backward contracting means that they are unstable sets.

Likewise we define the local stable sets of $f$ by
$$ W^s_{loc}(\omega,\theta):= \{ (x,\theta)\in X\times \T^d \colon x^+= \omega^+ \} .$$
These sets are mapped to each other and contracted by $f$.

Given $\varphi \in \Hscr_\alpha(X\times \T^d)$, 
we may define a $1$-dimensional additive cocycle
$F\colon X\times\T^d\times \R\to X\times\T^d\times \R$ by
$$F((x,\theta), t) :=(f(x,\theta), \Phi_{(x,\theta)}(t) ) := (f(x,\theta), t+\varphi(x,\theta)) , $$
whose iterates are given by 
$$F^n((x,\theta), t) =(f^n(x,\theta), \Phi^n_{(x,\theta)}(t)) = \left(f^n(x,\theta), \, t+S_n \varphi(x,\theta) \right) . $$

Because $f$ is partially hyperbolic, we can define fiber holonomies
 along unstable sets. These are families of maps
 $\left\{H^u_{(x,\theta), (y,\theta)} \colon \R\to\R \right\}$ 
 indexed by the pairs of points $(x,\theta)$, $(y,\theta)$ from the same
 local unstable sets $W^u_{loc}$. As usual these holonomies are 
defined by
$$H^u_{(x,\theta), (y,\theta)} (t):=\lim_{n\to\infty}  \Phi^{n}_{f^{-n}(y,\theta)} \, \left(\Phi_{f^{-n}(x,\theta)}^{n}\right)^{-1}  .$$
It is easy to see that these limits exist 
and take the form
$$H^u_{(x,\theta), (y,\theta)} (t) = t + h_{\varphi}^u((x,\theta),(y,\theta)), $$
where 
\begin{equation}
	\label{def hu}
h_{\varphi}^u((x,\theta),(y,\theta)):= \sum_{n=1}^{\infty}[\varphi(f^{-n}(y,\theta))-\varphi(f^{-n}(x,\theta))]  
\end{equation}
and the series converges by Weierstrass $M$-test, because
$$ \abs{\varphi(f^{-n}(y,\theta))-\varphi(f^{-n}(x,\theta))} \le v_\alpha(\varphi) \, 2^{-n\alpha} \, .
$$

We summarize below the usual properties of the holonomies
expressed in terms of the functions $h^u_{\varphi}$,
see for instance~\cite[Section 1.3]{AvV1}.

\begin{proposition}
	\label{holonomy properties}
	Given $(x,\theta) ,(y,\theta), (z,\theta)$ on the same $W^u_{loc}$, the following properties hold (the last one holds if $f(y,\theta)$ and $f (x,\theta)$ belong to the same $W^u_{loc}$):
	\begin{itemize}
		\item[(a)]  $ \displaystyle h^u_\varphi((x,\theta),(x,\theta)) = 0$,
		\smallskip
		
		\item[(b)]  $ \displaystyle h^u_\varphi((x,\theta),(y,\theta)) = - h^u_\varphi((y,\theta),(x,\theta)) $,
		\smallskip
		
		\item[(c)]  $\displaystyle  h^u_\varphi((x,\theta),(z,\theta)) = h^u_\varphi((x,\theta),(y,\theta))  +  h^u_\varphi((y,\theta),(z,\theta)) $,
		\smallskip
		
		\item[(d)]  $ \displaystyle h^u_\varphi((x,\theta),(y,\theta)) + \varphi(y,\theta) = \varphi(x,\theta)  + h^u_\varphi(f(x,\theta), f(y,\theta))$.
	\end{itemize}
\end{proposition}

Fix a future $p^+\in X^+$ and define a projection
$$
P \colon X\times \T^d \to X\times \T^d, \,P(\omega^-; \omega^+; \theta)=(\omega^-;p^+;\theta) .
$$
Notice that fixing $\theta\in \T^d$, any two points $x,y\in X$ are such that
$P(x,\theta)$ and $P(y,\theta)$ belong to the same local stable set $W^s_{loc}(p^+,\theta)$.

We define $\varphi^-:X\times\T^d\to \R$  for $a\in X\times\T^d$ by 
\begin{equation}
	\label{defvarphi-}
\varphi^-(a):=  
h_\varphi^u(a,P(a)) \,+ \, \varphi ( f^{-1}(a) ) \, + \, h_\varphi^u(P(f^{-1}(a)),f^{-1}(a)).
\end{equation} 
We can also define a $1$-dimensional additive
cocycle 
(see Figure~\ref{holonomy picture}) putting  for every $a\in X\times\T^d$,
$$ \Phi^-_a:= H^u_{a,P(a)} \, \Phi_{f^{-1}(a)}\, H^u_{P(f^{-1}(a)),f^{-1}(a)} . $$
The function $\varphi^-$ and the cocycle $\Phi^-$ are related by  $\Phi_a^-(t)=t+ \varphi^-(a)$.

\begin{figure}[h]
	\begin{center}
		\includegraphics[width=1.0\textwidth]{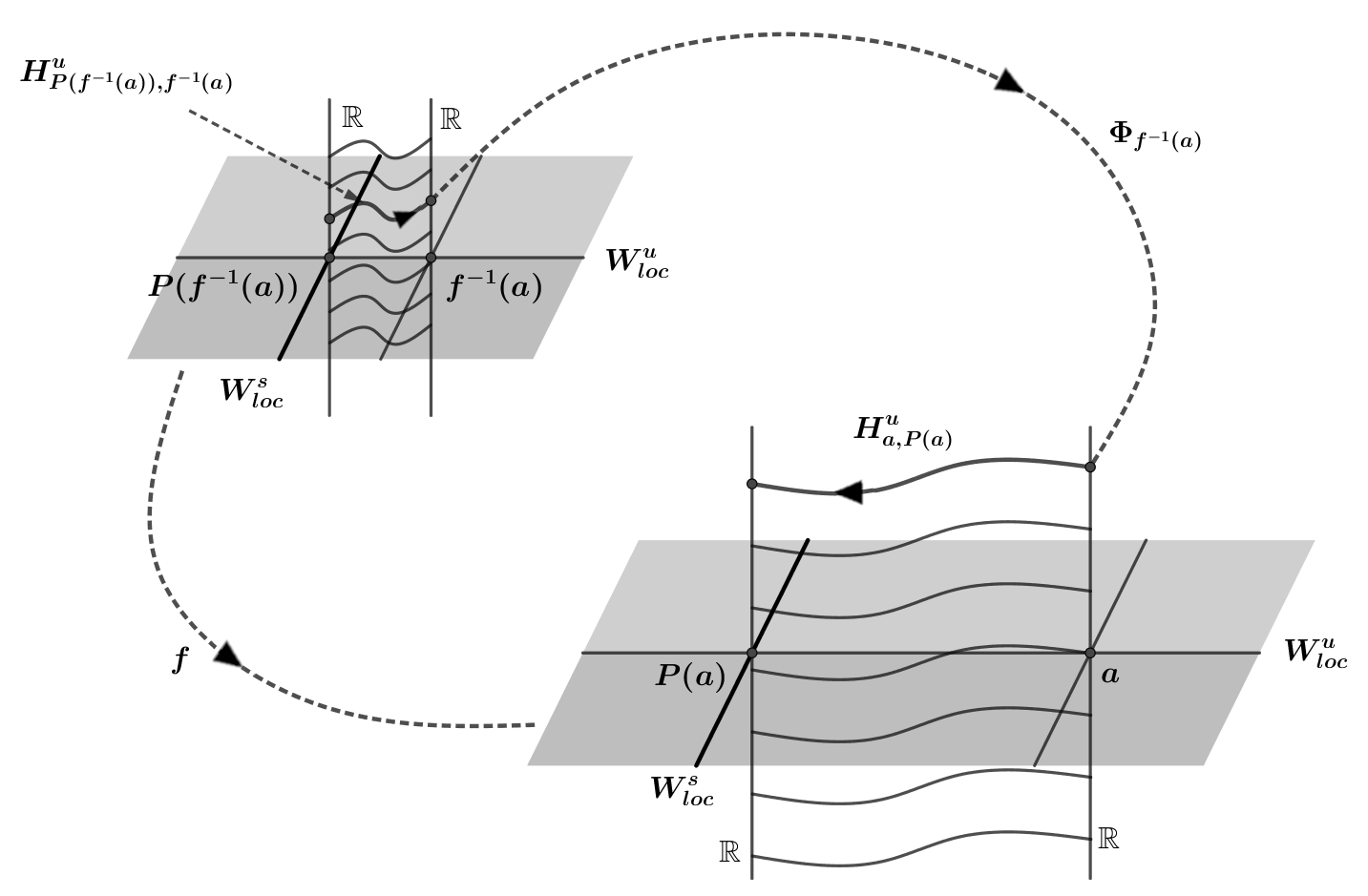}
		\caption{The cocycle $\Phi^-$.}
		\label{holonomy picture}
	\end{center}
\end{figure}

This dynamical interpretation and the properties of the holonomies readily imply that the cocycle $\Phi^-$ is constant along local unstable sets $W^u_{loc}$.
This implies that the functions $\Phi^-$ and $\varphi^-$ are future independent. Alternatively, by item (d) in Proposition~\ref{holonomy properties},
$$
h_\varphi^u(a,P(a))+\varphi\circ f^{-1}(a)=h_\varphi^u(f^{-1}(a),f^{-1}(P(a)))+\varphi \circ f^{-1}(P(a)),
$$ 
which implies the following representation of $\varphi^-$:
$$
\varphi^-(a)=h_\varphi^u(P \circ f^{-1}(a),f^{-1} \circ P(a))+\varphi \circ f^{-1} \circ P (a).
$$
Since $P$ fixes the future and $P$ appears in every term (and $P \circ f^{-1}(a)$ does not depend on non-negative coordinates), we conclude that $\varphi^-$ is future independent.

On the other hand by the definition in~\eqref{defvarphi-}  we have
\begin{align*}
\varphi ( f^{-1}(a) ) - \varphi^-(a)  &=  
h_\varphi^u(P(f^{-1}(a)),f^{-1}(a)) - 
h_\varphi^u(a,P(a))  .
\end{align*}
Thus the cohomological relation~\eqref{reduction} has the
solution 
$$\eta(a)=\eta_\varphi(a) := h_\varphi^u(a,P(a)).$$


We are now ready to prove Proposition~\ref{rtpast}.

\begin{proof}[Proof of  Proposition~\ref{rtpast}]
By definition~\eqref{def hu}, $h^u_\varphi$ depends linearly on $\varphi$ and hence the same is true about $\eta_\varphi$.

It remains to show  that for some $\beta > 0$, $\varphi^-\in \Hscr_\beta$ and $\Hscr_\alpha(X\times \T^d) \ni \varphi \mapsto \varphi^- \in \Hscr_\beta(X\times \T^d)$ is bounded. Recall that
$$
\varphi^-=\eta_\varphi-\eta_\varphi \circ f^{-1}+\varphi \circ f^{-1}.
$$
Since $f^{-1}$ is Lipschitz w.r.t. the distance $d$ on $X \times \T^d$,  it is enough to show that $\varphi \mapsto \eta_\varphi$ is bounded, that is, for some $\beta > 0$,
$
\norm{\eta_\varphi}_\beta \lesssim \norm{\varphi}_\alpha
$.

By definition,
$$
\norm{\eta_\varphi}_\beta=\norm{\eta_\varphi}_\infty+v_\beta^X(\eta_\varphi)+v_\beta^{\T^d}(\eta_\varphi).
$$

We have already shown that $\norm{\eta_\varphi}_\infty\lesssim v_\alpha^X(\varphi) \leq \norm{\varphi}_\alpha$.
Rewrite
$$
\eta_\varphi(\omega,\theta)=\sum_{n=1}^{\infty} \left[ \varphi \circ f^{-n}\circ P(\omega,\theta)-\varphi \circ f^{-n}(\omega,\theta) \right] =:\sum_{n=1}^{\infty}g_n(\omega,\theta).
$$

We want to show that for some $\beta > 0$ and for any $\theta\in \T^d$ and any $k\in \N$, if $x, y \in X$ are such that $x_j=y_j$ for all indices $j$ with $\abs{j}\leq k$ (meaning that $d(x, y) \le 2^{-(k+1)}$), then 
$$
\abs{\eta_\varphi(x,\theta)-\eta_\varphi(y,\theta)}\lesssim v_\alpha^X(\varphi)2^{-k\beta} \, .
$$
This would imply that $v_\beta^X(\eta_\varphi)\lesssim v_\alpha^X(\varphi)\leq \norm{\varphi}_\alpha$.

Indeed, by the triangle inequality we have
$$
\abs{\eta_\varphi(x,\theta)-\eta_\varphi(y,\theta)}\leq \sum_{1 \le n \le \frac{k}{2}} \abs{g_n(x,\theta)-g_n(y,\theta)}+\sum_{n>\frac{k}{2}}\abs{g_n(x,\theta)-g_n(y,\theta)}.
$$

We analyze the two sums on the  right-hand side separately.
\begin{align*}
\sum_{n>\frac{k}{2}}\abs{g_n(x,\theta)-g_n(y,\theta)} & \leq \sum_{n>\frac{k}{2}}\abs{g_n(x,\theta)}+\sum_{n>\frac{k}{2}}\abs{g_n(y,\theta)} \\
& \lesssim \sum_{n>\frac{k}{2}}v_\alpha^X(\varphi) 2^{-n\alpha}  \lesssim v_\alpha^X(\varphi)2^{- \frac{k}{2} \, \alpha} \, .
\end{align*}

In order to estimate the sum
$
\sum_{1 \le n \le \frac{k}{2}} \abs{g_n(x,\theta)-g_n(y,\theta)}
$, note that 
since $k \ge 2n$, we have 
$$d(\sigma^{-n}x,\sigma^{-n}y)\leq 2^{-(k-n)} \quad \text{and so} \quad d (f^{-n}(x,\theta),f^{-n}(y,\theta))\leq 2^{-(k-n)}. $$

Then for $\varphi \in \Hscr_\alpha (X \times \T^d)$,
$$
\abs{\varphi \circ f^{-n}(x,\theta)-\varphi \circ f^{-n}(y,\theta)}\leq v_\alpha^X(\varphi)2^{-(k-n)\alpha} \, .
$$
The same estimate clearly also holds with $P(x,\theta), P(y,\theta)$ instead of $(x,\theta), (y,\theta)$, and we conclude that
$$
\abs{g_n(x,\theta)-g_n(y,\theta)}\lesssim v_\alpha^X(\varphi) 2^{-(k-n)\alpha} .
$$
Then
\begin{align*}
\sum_{1 \le n \le \frac{k}{2}} \abs{g_n(x,\theta)-g_n(y,\theta)} & \lesssim v_\alpha^X(\varphi) \sum_{1 \le n \le \frac{k}{2}}  2^{-(k-n)\alpha} \\
& \lesssim v_\alpha^X(\varphi)\sum_{ j \ge \frac{k}{2}} 2^{-j \alpha}  \lesssim v_\alpha^X(\varphi)2^{- \frac{k}{2} \, \alpha} \, .
\end{align*}

Combining the two estimates, we have
$$
\abs{\eta_\varphi(x,\theta)-\eta_\varphi(y,\theta)}\lesssim v_\alpha^X(\varphi) 2^{- k \, \frac{\alpha}{2}}
$$
for every $\theta\in \T^d$ and every $k\in \N$, which further implies, with $\beta = \frac{\alpha}{2}$, 
$$
v_\beta^X(\eta_\varphi)\lesssim v_\alpha^X(\varphi)\leq \norm{\varphi}_\alpha.
$$

Following the same strategy, we also get that $v_\beta^{\T^d}(\eta_\varphi)\lesssim \norm{\varphi}_\alpha$ (with any $\beta < \alpha$). More precisely, fix any $x\in X$ and for any given $\theta,\theta'\in \T^d$, let $N\in \N$ be such that $\abs{\theta-\theta'}\asymp 2^{-N}$, so $N\asymp \log \frac{1}{\abs{\theta-\theta'}}$.
Like before, 
$$
\sum_{n\geq N}\abs{g_n(x,\theta)}\leq \sum_{n\geq N} v_\alpha^X(\varphi) 2^{-n\alpha}\lesssim v_\alpha^X(\varphi) 2^{-N\alpha}\lesssim v_\alpha^X(\varphi)\abs{\theta-\theta'}^\alpha.
$$
The same estimate holds for $(x,\theta')$. Therefore,
$$
\sum_{n\geq N}\abs{g_n(x,\theta)-g_n(x,\theta')}\lesssim v_\alpha^X(\varphi)\abs{\theta-\theta'}^\alpha\leq \norm{\varphi}_\alpha \abs{\theta-\theta'}^\alpha.
$$

On the other hand, 
$$
\abs{\varphi \circ f^{-n}(x,\theta)-\varphi \circ f^{-n}(x,\theta')}\leq v_\alpha^{\T^d}(\varphi) \abs{\theta-\theta'}^\alpha
$$
and the same estimate also holds for $P(x,\theta), P(x,\theta')$. Then $\forall n\in\N$,
$$
\abs{g_n(x,\theta)-g_n(x,\theta')}\lesssim v_\alpha^{\T^d}(\varphi) \abs{\theta-\theta'}^\alpha ,
$$
which implies
\begin{align*}
\sum_{0\leq n< N}\abs{g_n(x,\theta)-g_n(x,\theta')} & 
\lesssim v_\alpha^{\T^d}(\varphi) \abs{\theta-\theta'}^\alpha \, N \\
&  \kern-3em
\lesssim v_\alpha^{\T^d}(\varphi) \abs{\theta-\theta'}^\alpha \log \frac{1}{\abs{\theta-\theta'}} 
 \le \norm{\varphi}_\alpha \abs{\theta-\theta'}^{\beta} \, .
\end{align*}

Then
$$
\abs{\eta_\varphi(x,\theta)-\eta_\varphi(x,\theta')} \lesssim  \norm{\varphi}_\alpha \abs{\theta-\theta'}^\beta ,
$$
showing that $v_\beta^{\T^d}(\eta_\varphi)\lesssim \norm{\varphi}_\alpha$ and completing the proof of Proposition~\ref{rtpast}.
\end{proof}

\medskip

\subsection*{Acknowledgments}

All three authors were supported by FCT-Funda\c{c}\~{a}o para a Ci\^{e}ncia e a Tecnologia through the project  PTDC / MAT-PUR / 29126 / 2017.
The first author was also supported by a FAPERJ postdoctoral grant. The third author was also supported by the CNPq research grant 313777/2020-9 and  by the Coordena\c{c}\~ao de Aperfei\c{c}oamento de Pessoal de N\'ivel Superior - Brasil (CAPES) - Finance Code 001.

\bigskip

\bibliographystyle{amsplain} 
\bibliography{references}

\providecommand{\bysame}{\leavevmode\hbox to3em{\hrulefill}\thinspace}
\providecommand{\MR}{\relax\ifhmode\unskip\space\fi MR }
\providecommand{\MRhref}[2]{%
  \href{http://www.ams.org/mathscinet-getitem?mr=#1}{#2}
}
\providecommand{\href}[2]{#2}
\begin{thebibliography}{10}

\bibitem{Alves-Freitas-L-V}
Jos\'{e}~F. Alves, Jorge~M. Freitas, Stefano Luzzatto, and Sandro Vaienti,
  \emph{From rates of mixing to recurrence times via large deviations}, Adv.
  Math. \textbf{228} (2011), no.~2, 1203--1236.

\bibitem{AvV1}
A.~Avila and M.~Viana, \emph{Simplicity of {L}yapunov spectra: a sufficient
  criterion}, Port. Math. \textbf{64} (2007), 311--376.

\bibitem{Baladi}
Viviane Baladi, \emph{Positive transfer operators and decay of correlations},
  Advanced Series in Nonlinear Dynamics, vol.~16, World Scientific Publishing
  Co., Inc., River Edge, NJ, 2000.

\bibitem{Borda}
Bence Borda, \emph{Equidistribution of random walks on compact groups {II}.
  {T}he {W}asserstein metric}, Bernoulli \textbf{27} (2021), no.~4, 2598--2623.

\bibitem{CDK-paper1}
Ao~Cai, Pedro Duarte, and Silvius Klein, \emph{Mixed random-quasiperiodic
  cocycles}, Bull. Braz. Math. Soc. (N.S.) \textbf{53} (2022), no.~4,
  1469--1497.

\bibitem{Chazottes-concentration}
Jean-Ren{\'e} Chazottes and S{\'e}bastien Gou{\"e}zel, \emph{Optimal
  concentration inequalities for dynamical systems}, Communications in
  Mathematical Physics \textbf{316} (2012), no.~3, 843--889.

\bibitem{slow-mixing}
J\'{e}r\^{o}me Dedecker, S\'{e}bastien Gou\"{e}zel, and Florence Merlev\`ede,
  \emph{Large and moderate deviations for bounded functions of slowly mixing
  {M}arkov chains}, Stoch. Dyn. \textbf{18} (2018), no.~2, 1850017, 38.

\bibitem{DMPU2009}
J\'er\^ome Dedecker, Florence Merlev\`ede, Magda Peligrad, and Sergey Utev,
  \emph{Moderate deviations for stationary sequences of bounded random
  variables}, Ann. Inst. Henri Poincar\'e{} Probab. Stat. \textbf{45} (2009),
  no.~2, 453--476.

\bibitem{Dedecker-Prieur}
J{\'e}r{\^o}me Dedecker and Cl{\'e}mentine Prieur, \emph{New dependence
  coefficients. {Examples} and applications to statistics}, Probab. Theory
  Relat. Fields \textbf{132} (2005), no.~2, 203--236.

\bibitem{Liverani}
Mark~F. Demers, Niloofar Kiamari, and Carlangelo Liverani, \emph{Transfer
  operators in hyperbolic dynamics---an introduction}, 33$^{\rm o}$
  Col\'{o}quio Brasileiro de Matem\'{a}tica, Instituto Nacional de
  Matem\'{a}tica Pura e Aplicada (IMPA), Rio de Janeiro, 2021.

\bibitem{DF-CLT}
Dmitry Dolgopyat and Bassam Fayad, \emph{Limit theorems for toral
  translations}, Hyperbolic dynamics, fluctuations and large deviations, Proc.
  Sympos. Pure Math., vol.~89, Amer. Math. Soc., Providence, RI, 2015,
  pp.~227--277.

\bibitem{DK-book}
Pedro Duarte and Silvius Klein, \emph{{L}yapunov exponents of linear cocycles;
  continuity via large deviations}, Atlantis Studies in Dynamical Systems,
  vol.~3, Atlantis Press, 2016.

\bibitem{DKP}
Pedro Duarte, Silvius Klein, and Mauricio Poletti, \emph{H\"{o}lder continuity
  of the {L}yapunov exponents of linear cocycles over hyperbolic maps}, Math.
  Z. \textbf{302} (2022), no.~4, 2285--2325.

\bibitem{Hoeffding-Markov}
Peter~W. Glynn and Dirk Ormoneit, \emph{Hoeffding's inequality for uniformly
  ergodic {M}arkov chains}, Statist. Probab. Lett. \textbf{56} (2002), no.~2,
  143--146.

\bibitem{Gordin-L}
M.~I. Gordin and B.~A. Lif\v{s}ic, \emph{Central limit theorem for stationary
  {M}arkov processes}, Dokl. Akad. Nauk SSSR \textbf{239} (1978), no.~4,
  766--767.

\bibitem{Gordin2}
Mikhail Gordin and Hajo Holzmann, \emph{The central limit theorem for
  stationary {M}arkov chains under invariant splittings}, Stoch. Dyn.
  \textbf{4} (2004), no.~1, 15--30.

\bibitem{HH}
Hubert Hennion and Lo{\"{\i}}c Herv{\'e}, \emph{Limit theorems for {M}arkov
  chains and stochastic properties of dynamical systems by quasi-compactness},
  Lecture Notes in Mathematics, vol. 1766, Springer-Verlag, Berlin, 2001.

\bibitem{KLM}
Silvius Klein, Xiao-Chuan Liu, and Aline Melo, \emph{Uniform convergence rate
  for {B}irkhoff means of certain uniquely ergodic toral maps}, Ergodic Theory
  Dynam. Systems \textbf{41} (2021), no.~11, 3363--3388.

\bibitem{CLT-MW}
Michael Maxwell and Michael Woodroofe, \emph{Central limit theorems for
  additive functionals of {Markov} chains.}, Ann. Probab. \textbf{28} (2000),
  no.~2, 713--724 (English).

\bibitem{Melbourne-PAMS}
Ian Melbourne, \emph{Large and moderate deviations for slowly mixing dynamical
  systems}, Proc. Amer. Math. Soc. \textbf{137} (2009), no.~5, 1735--1741.

\bibitem{Melbourne-Nicol}
Ian Melbourne and Matthew Nicol, \emph{Large deviations for nonuniformly
  hyperbolic systems}, Trans. Amer. Math. Soc. \textbf{360} (2008), no.~12,
  6661--6676.

\bibitem{Markov-chains-book}
Sean Meyn and Richard~L. Tweedie, \emph{Markov chains and stochastic
  stability}, second ed., Cambridge University Press, Cambridge, 2009, With a
  prologue by Peter W. Glynn.

\bibitem{Grisha1}
Grigorii Monakov, \emph{Ergodic theorem for nonstationary random walks on
  compact abelian groups}, Proc. Amer. Math. Soc. \textbf{152} (2024), no.~9,
  3855--3866.

\bibitem{Peligrad}
Magda Peligrad, Sergey Utev, and Wei~Biao Wu, \emph{A maximal
  {{\(\mathbb{L}_{p}\)}}-inequality for stationary sequences and its
  applications}, Proc. Am. Math. Soc. \textbf{135} (2007), no.~2, 541--550.

\bibitem{Anselmo-dissertation}
Anselmo Pontes, \emph{Limit laws for dynamical systems with some
  hyperbolicity}, Master's thesis, Pontifical Catholic University of Rio de
  Janeiro (PUC-Rio), 2024.

\bibitem{LSYoung-annals}
Lai-Sang Young, \emph{Statistical properties of dynamical systems with some
  hyperbolicity}, Ann. of Math. (2) \textbf{147} (1998), no.~3, 585--650.
  \MR{1637655}

\bibitem{LSYoung99}
\bysame, \emph{Recurrence times and rates of mixing}, Israel J. Math.
  \textbf{110} (1999), 153--188.

\end{thebibliography}

\end{document}